\documentclass[10pt,a4paper]{article}
\usepackage{amsmath}
\usepackage{amssymb}
\usepackage{amsthm}
\usepackage{mathrsfs}
\usepackage{enumitem}
\newtheorem{theorem}{Theorem}[section]
\newtheorem{definition}[theorem]{Definition}
\newtheorem{lemma}[theorem]{Lemma}
\newtheorem{example}[theorem]{Example}
\newtheorem{remark}[theorem]{Remark}
\newcommand{\sind}{k}
\newcommand{\epi}{\mathop{}\!\mathrm{epi\hspace{0.05 cm}}}
\newcommand{\gr}{\mathop{}\!\mathrm{Gr\hspace{0.05 cm}}} 
\newcommand{\dif}{\mathop{}\!\mathrm{d}} 
\newcommand{\co}{\mathop{}\!\mathrm{co\hspace{0.05 cm}}}
\newcommand{\wl}{\mathop{}\!\mathrm{w}\hspace{0.03 cm}\text{-}\hspace{-0.1 cm}}
\newcommand{\sqa}{\varepsilon}
\newcommand{\sqb}{\eta}
\newcommand{\sqc}{\theta}
\newcommand{\sqd}{\vartheta}
\newcommand{\sqe}{\nu}

\begin{document}

\title{A New Constraint Qualification for Mixed Constrained Optimal Control}

\author{
Rodrigo B. Moreira\footnote{rbmoreira@uesc.br} \\
Department of Exact Sciences, \\
State University of Santa Cruz (UESC), \\
Ilhéus, Bahia, Brazil
\and
Valeriano A. de Oliveira\footnote{valeriano.oliveira@unesp.br} \\
Department of Mathematics, \\
S\~ao Paulo State University (UNESP), \\
S\~ao Jos\'e do Rio Preto, S\~ao Paulo, Brazil
}

\date{}

\maketitle

\begin{abstract}
In recent developments, a novel set of necessary optimality conditions for mixed constrained optimal control problems, termed the asymptotic weak maximum principle, has been formulated. These novel conditions deviate from the classical ones by virtue of their sequential nature and the fact that they are satisfied regardless of the regularity conditions imposed on the mixed constraints. Furthermore, due to their asymptotic behaviour, these conditions serve as a precise tool for use as stopping criteria in numerical methods of solution. However, it should be noted that, in certain instances, these conditions may not be sufficiently robust to fully characterize optimal solutions, as they can be satisfied by processes that are not extremals. The present study proposes a novel constraint qualification, meticulously developed to address these asymptotic optimality conditions. It is demonstrated that the asymptotic weak maximum principle implies the classical weak maximum principle when the newly proposed constraint qualification is verified. It is further demonstrated that, in the smooth setting, this constraint qualification is the weakest one that possesses such a property. Additionally, this study present sufficient criteria for the validity of the newly proposed constraint qualification.

\end{abstract}


\section{Introduction}

In the context of nonlinear optimization problems, Andreani et al. \cite{Andreani2011} established the approximate KKT (AKKT) and approximate projected gradient (AGP) sequential optimality conditions. They demonstrated that AKKT and AGP are necessarily satisfied at every local optimal solution and that popular methods, such as the augmented Lagrangian method, generate sequences whose limit points satisfy AKKT. However, as pointed out by  Haeser \cite{Haeser2016}, a point may satisfy AKKT without being optimal or without satisfying KKT. This observation led to the study of constraint qualifications (CQ) associated with sequential conditions. In their work, Andreani et al. \cite{Andreani2016} introduced the Cone Continuity Property (CCP), which is the weakest condition under which AKKT implies KKT. For AGP, the minimum condition that guarantees KKT is called AGP-regularity and is defined in \cite{Andreani2018}. Following the work of Andreani et al. \cite{Andreani2011}, B\"{o}rgens et al.  \cite{Borgens2020} extended the concept of AKKT conditions to problems posed in infinite-dimensional spaces, specifically in real Banach spaces, and proposed the AKKT regularity constraint qualification. Their results are the most general among a number of related works, including \cite{Borgens2019, Kanzow2018, Steck2018}.

In the context of optimal control theory, the study of sequential-type optimality conditions was initiated by Moreira and de Oliveira in \cite{Moreira2024art}. In their work, Moreira and de Oliveira present a version of the maximum principle in sequential terms for optimal control problems with mixed constraints. As a result, they establish a set of optimality conditions referred to as the Asymptotic Weak Maximum Principle (AWMP). The problem studied in \cite{Moreira2024art} is the same as Problem (P), which is given by
\begin{subequations}
	\begin{align}
		\nonumber
		&\text{minimize } \ \ell(x(t_{0}), x(t_{1}))\\
		\nonumber
		&\text{over } \ (x, u) \in W^{1,1}([t_{0}, t_{1}]; \mathbb{R}^{n})\times L^{1}([t_{0}, t_{1}]; \mathbb{R}^{m}) \text{ satisfying}\\
		&\hspace{0.9 cm} \dot{x}(t) = f(t, x(t), u(t)) \text{ a.e. in } [t_{0}, t_{1}],
		\label{EqP_din}\\
		&\hspace{0.9 cm}0 = b(t, x(t), u(t)) \text{ a.e. in } [t_{0}, t_{1}],
		\label{EqP_Rigu}\\
		&\hspace{0.9 cm}0 \geq g(t, x(t), u(t)) \text{ a.e. in } [t_{0}, t_{1}],
		\label{EqP_Rdesi}\\
		&\hspace{0.9 cm}u(t) \in  U(t) \text{ a.e. in } [t_{0}, t_{1}],
		\label{EqP_Rmist}\\
		&\hspace{0.9 cm}(x(t_{0}), x(t_{1})) \in C, 
		\label{EqP_Rendp}
	\end{align}
\end{subequations}
where the interval $[t_{0}, t_{1}]$, the functions $\ell: \mathbb{R}^{n}\times\mathbb{R}^{n} \rightarrow \mathbb{R}$, $f: [t_{0}, t_{1}]\times\mathbb{R}^{n}\times\mathbb{R}^{m} \rightarrow \mathbb{R}^{n}$, $b: [t_{0}, t_{1}]\times \mathbb{R}^{n}\times\mathbb{R}^{m} \rightarrow \mathbb{R}^{m_{b}}$, and $g: [t_{0}, t_{1}]\times\mathbb{R} ^{n}\times\mathbb{R}^{m} \rightarrow \mathbb{R}^{m_{g}}$, the non-empty set-valued map $U: [t_{0}, t_{1}] \rightsquigarrow \mathbb{R}^{m}$ and the set $C \subset \mathbb{R}^{n} \times \mathbb{R}^{n}$ are given.

Although AWMP constitutes a genuine set of necessary optimality conditions that are very useful in the design of numerical methods of solution, it may be satisfied by processes that are not extremals. The main goal of this work is to present a new constraint qualification that has the property of being the weakest one under which AWMP implies the classical weak maximum principle. Drawing inspiration from recent studies in the area of nonlinear programming, particularly the cone continuity property (CCP) \cite{Andreani2011}, the asymptotically Mordukhovich-regularity (AM-regularity) \cite{Mehlitz2020}, and the AKKT-regularity \cite{Borgens2020}, we propose this novel constraint qualification, which we term AWMP-regularity. We demonstrate that if the problem is AWMP-regular, then every local optimal solution satisfies the weak maximum principle. 

It is widely acknowledged that constraint qualifications play a pivotal role in ensuring the validity of the maximum principle in optimal control problems with mixed constraints. One of the most well-known of these is the interiority condition, as presented in  de Pinho et al. \cite{dePinho2001}. It should be noted that the proof of the maximum principle for Problem (P) under this condition requires an additional assumption: the convexity of the set $\{(f(t, x, u), b(t, x, u), g(t, x, u)) : u \in U(t)\}$ for each $t \in [t_{0}, t_{1}]$. In contrast, the formulation of AKKT-regularity does not impose this kind of assumption.

The Weak Basic Constraint Qualification (WBCQ) was introduced by Li and Ye in \cite{Li2016} for autonomous problems. In deriving the maximum principle under WBCQ, additional assumptions must be made, such as the calmness condition and the compactness of a specific set defined along the optimal process.

Other important constraint qualifications are those of the full-rank type. There are different versions in the literature, among which we highlight the approaches adopted by Hestenes \cite{Hestenes1966}, Neustadt \cite{Neustadt1976}, Osmolovskii \cite{Osmolovskii1975}, de Pinho and Ilchmann \cite{dePinho2002}, and de Pinho \cite{dePinho2003}. In addition to full-rank constraint qualifications, there are also constant-rank types, such as proposed in Andreani et al. in \cite{Andreani2020}.

The well-known Mangasarian-Fromovitz condition also ensures the validity of the maximum principle, as demonstrated by Dmitruk in \cite{Dmitruk1993} and Pinho and Rosenblueth in \cite{dePinho2008}. More recently, it has been reformulated as the Calibrated Constraint Qualification (CCQ), which is also applicable to nonsmooth problems and more general formulations than (P), as shown by Clarke and de Pinho in \cite{Clarke2010}.

The extant literature also encompasses other constraint qualifications, such as those presented by Arutyunov and Karamzin in \cite{Arutyunov2015} and by Arutyunov et al. in \cite{Arutyunov2010}. In both works, the authors argue that the notion of regularity they propose in fact implies the Robinson Constraint Qualification \cite[p. 426]{Mordukhovich2006}. 

The relationship between CCP, AKKT-regularity, AM-regularity, and classical constraint qualifications has been established in the literature (see, \cite{Andreani2016,Borgens2020,Mehlitz2020}). Here, however, we do not relate AWMP-regularity to any of the constraint qualifications found in the literature on mixed constrained optimal control problems. Instead, we propose an asymptotic version of the calibrated constraint qualification and demonstrate that it implies AWMP-regularity.

The sequential optimality conditions provided by the asymptotic weak maximum principle \cite{Moreira2024art} differ from other traditional necessary optimality conditions found in previous work from the literature. The aforementioned conditions are given in sequential terms and fit precisely to sequences of control processes generated by numerical methods, such as the method of multipliers proposed in the same article. In summary, these optimality conditions possess significant practical potential, as there exist tools capable of generating sequences that obey them. Furthermore, AWMP-regularity establishes minimal conditions under which the sequential conditions imply the classical optimality conditions furnished by the weak maximum principle. Consequently, the asymptotic weak maximum principle provides a theoretical framework that can be effectively utilized to identify optimal solutions for control problems with mixed constraints. When the AWMP-regularity condition is satisfied, this set of solutions is filtered with the same precision as the weak maximum principle.

The present work is structured as follows. In Section \ref{Sec_Prel}, the main definitions, notations, and some important results that will be used throughout the text are introduced. In Section \ref{Sec_New_CQ}, the AWMP-regularity is properly defined and shown to be a constraint qualification. The Section \ref{Suff_Crit} provides sufficiency criteria for the validity of AWMP-regularity. The Section \ref{Sec_Concl} presents a summary of conclusions drawn from the study.

\section{Preliminaries} \label{Sec_Prel}

In this work, we denote strong, weak, and weak$^\star$ convergence by $\to$, $\rightharpoonup$, and $\stackrel{\star}{\rightharpoonup}$, respectively. 

The closed unit ball centered at the origin in a Euclidean space is denoted by $B$ (regardless of the dimension). A Euclidean norm is denoted by $\vert\cdot\vert$.

We will use $A^{T}$ to denote the transpose of a matrix $A$. Given $a \in \mathbb{R}$, the $n$-dimensional column vector with entries all set equal to $a$ is denoted by $\mathbf{a}_{n \times 1}$.

We write $L^{p}_{l}$ with $1 \leq p \leq \infty$ to denote the classical spaces $L^{p}([t_{0}, t_{1}]; \mathbb{R}^{l})$ and $W_{n}^{1,1}$ for the space of absolutely continuous functions $W^{1,1}([t_{0}, t_{1}];\mathbb{R}^{n})$.

Given $(t,x,u) \in [t_0,t_1] \times \mathbb{R}^n \times \mathbb{R}^m$, we denote $g_{j}^{-}(t, x, u) := \max\{-g_{j}(t, x, u), 0\}$, $j=1,\ldots,m_g$.

The augmented Pontryagin function $\mathcal{H}: [t_{0}, t_{1}] \times \mathbb{R}^{n} \times \mathbb{R}^{n} \times \mathbb{R}^{m_{b}} \times \mathbb{R}^{m_{g}} \times \mathbb{R}^{m} \rightarrow \mathbb{R}$ associated with (P) is defined by
\begin{align*}
\mathcal{H}(t, x, p, q, r, u) := p\cdot f(t, x, u) + q\cdot b(t, x, u)  + r\cdot g(t, x, u).
\end{align*}

In addition to the notation introduced earlier, we will also briefly review some key definitions and results that are essential for this paper. First, recall from Vinter \cite[p. 202]{Vinter2000} that a \emph{process} is defined as a pair $(x, u)$, where $u \in L^{1}_{m}$ represents a control and $x \in W_{n}^{1,1}$ is an arc that solves \eqref{EqP_din}. A process $(\bar{x}, \bar{u})$ that satisfies \eqref{EqP_Rigu}, \eqref{EqP_Rdesi}, \eqref{EqP_Rmist}, and \eqref{EqP_Rendp} is termed a \emph{feasible process}. A feasible process $(\bar{x}, \bar{u})$ is called a \emph{weak local minimum} of (P) if there exists $\delta > 0$ such that 
$$
\ell(\bar{x}(t_{0}), \bar{x}(t_{1})) \leq \ell(x(t_{0}), x(t_{1}))
$$ 
for all feasible processes $(x, u)$ satisfying 
$$
(x(t), u(t)) \in \Omega_{\delta}(t) := \{ (y, z) \in \mathbb{R}^{n} \times \mathbb{R}^{m} : |(y, z) - (\bar{x}(t), \bar{u}(t))| \leq \delta \} ~ \text{a.e in} ~ [t_{0}, t_{1}].
$$ 
A feasible process $(\bar{x}, \bar{u})$ is \emph{stationary} for Problem (P) if it satisfies the weak maximum principle conditions (see de Pinho and Vinter \cite[Theorem 3.2]{dePinho1995}, for example; see also Theorem \ref{Teo_PMF_Versao_Class} below).

This paper employs a number of standard concepts from nonsmooth analysis. Here, we merely provide a concise summary of the key definitions; for a more exhaustive presentation, please refer to \cite{Clarke1998, Clarke1983, Vinter2000}.

Let $S \subset \mathbb{R}^{n}$ be a closed set, and let $\bar{x} \in S$. A vector $v \in \mathbb{R}^{n}$ is called a \textit{limiting normal} to $S$ at $\bar{x}$ if there exist sequences $v^{\sind} \rightarrow v$, $x^{\sind} \rightarrow \bar{x}$, with $x^{\sind} \in S$, and $\{M_{\sind}\} \subset \mathbb{R}$, with $M_\sind \geq 0$, such that 
\begin{displaymath}
v^{\sind}\cdot(x - x^{\sind}) \leq M_{\sind} |x - x^{\sind}|^{2} ~ \forall \, x \in S, ~ \forall \, \sind \in \mathbb{N}.
\end{displaymath}
The set of all limiting normals to $S$ at $\bar{x}$, denoted $\mathcal{N}_{S}(\bar{x})$, forms a (possibly nonconvex) cone known as the \emph{limiting normal cone} to $S$ at $\bar{x}$, or the \emph{Mordukhovich normal cone}.

Consider a lower semicontinuous function $f: \mathbb{R}^{n} \rightarrow \mathbb{R} \cup \{+\infty\}$, and let $\bar{x} \in \mathbb{R}^{n}$ with $f(\bar{x}) < +\infty$. The \emph{limit subdifferential} of $f$ at $\bar{x}$ is defined as
\begin{align*}
	\partial f(\bar{x}) := \left\{\xi : (\xi, -1) \in \mathcal{N}_{\epi f}(\bar{x}, f(\bar{x}))\right\},
\end{align*}
where $\epi f := \{(\eta, x) : \eta \geq f(x)\}$ is the \emph{epigraph of $f$}. The limit subdifferential of $f$ at $x$ is also known as the \emph{Mordukhovich subdifferential}.

Let $(\bar{x}, \bar{u})$ be a reference process for (P). Given a parameter $\delta > 0$, we will consider the following fundamental assumptions:
\begin{itemize}
	\item[(H1)] The function $f(\cdot, x, u)$ is Lebesgue measurable for each $(x, u)$ and there exists an integrable function $k_{f}$ such that, almost everywhere in $[t_{0}, t_{1}]$,
	\begin{align*}
		|f(t, x, u) - f(t, x^\prime, u^\prime)| \leq k_{f}(t) |(x, u) - (x^\prime, u^\prime)| ~ \forall \, (x, u), \, (x^\prime, u^\prime) \in \Omega_\delta(t).
	\end{align*}

	\item[(H2)] The graph of $U$, denoted by $\gr(U) := \{(t,v) : t \in [t_{0}, t_{1}] \text{ and } v \in U(t)\}$, is a Borel measurable set and $U_{\delta}(t) := (\bar{u}(t) + \delta B)\cap U(t)$ is closed and convex almost everywhere in $[t_{0}, t_{1}]$.

	\item[(H3)] The set $C$ is closed and the function $\ell$ is locally Lipschitz in a neighborhood of $(\bar{x}(t_{0}), \bar{x}(t_{1}))$.

	\item[(H4)] The function $(b, g)(\cdot, x, u)$ is Lebesgue measurable for each $(x,u)$.

	\item[(H5)] The function $(b, g)(t, \cdot, \cdot)$ is continuously differentiable on $(\bar{x}(t), \bar{u}(t)) + \delta B$ almost everywhere in $[t_{0}, t_{1}]$ and there exists an integrable function $k_{bg}$ such that, almost everywhere in $[t_{0}, t_{1}]$,
	\begin{align*}
		|(b, g)(t, x, u) - (b, g)(t, x^{\prime}, u^{\prime})| \leq k_{bg}(t)|(x, u) - (x^{\prime}, u^{\prime})|
	\end{align*}
    for all $(x, u), \, (x^{\prime}, u^{\prime}) \in \Omega_{\delta}(t)$.

    \item[(H6)]	The functions $K_{f}(t) := |f( t, \bar{x}(t), \bar{u}(t))|$ and $K_{bg}(t) := |(b, g)(t , \bar{x}(t), \bar{u}(t))|$ are integrable over $[t_{0}, t_{1}]$ and $\bar{u} \in L_{m}^{\infty}$.
    
    \item[(H7)] Let  
\[
D := \left\{ (x, v) \in W_{n}^{1,1} \times L_{m}^{2} : (x(t), v(t)) \in \Omega_\delta(t) ~ \text{a.e. in} ~ [t_{0}, t_{1}] \right\}.
\]
The mappings $\mathcal{F}: D \rightarrow L_{n}^{\infty}$ and $\mathcal{P}: D \rightarrow \mathbb{R}$ defined respectively by
\[
t \mapsto \mathcal{F}(x, v)(t) := \int_{t_{0}}^{t} f(s, x(s), v(s)) \dif s  ~ \text{a.e. in} ~ [t_{0}, t_{1}]
\] 
and
\begin{align*}
\mathcal{P}(x, v) := \int_{t_{0}}^{t_{1}} \left[\sum_{i=1}^{m_{b}} (b_{i}(t, x(t), v(t)))^{2} + \sum_{j=1}^{m_{g}} (\max \{g_{j}(t, x(t), v(t)), 0\})^{2} \right] \dif t
\end{align*}
are completely continuous and weakly sequentially lower semicontinuous, respectively.
\end{itemize}

Assuming the validity of (H1)--(H6), and under certain other hypotheses such as a constraint qualification on the mixed constraints, see references \cite{Andreani2020, Arutyunov2000, dePinho2002, dePinho2009, dePinho2008, Dmitruk1993, Pereira2020, Schwarzkopf1976}, the classical weak maximum principle is presented as follows.

\begin{theorem}[Classical Weak Maximum Principle (WMP)]
\label{Teo_PMF_Versao_Class} 
Let $(\bar{x}, \bar{u})$ be a weak local minimizer for Problem (P). Then there exist $\lambda \in [0, +\infty)$, $p \in W_{n}^{1,1}$, $q \in L_{m_{b}}^{1}$, $r \in L_{m_{g}}^{1}$, and $\zeta \in L_{m}^{1}$ such that 
\begin{align*} 
    &\lambda + \Vert p\Vert_{L_{n}^{\infty}} \neq 0,\\
    &(-\dot{p} (t), \zeta(t)) \in \co\partial_{x, u} \mathcal{H}(t, \bar{x}(t), p(t), q(t), r(t), \bar{u}(t)) ~ \text{a.e.~in} ~ [t_{0}, t_{1}],\\ 
    &\zeta(t) \in \co\mathcal{N}_{U(t)}(\bar{u}(t)) ~ \text{a.e.~in} ~ [t_{0}, t_{1}],\\
    &r(t)\cdot g(t, \bar{x}(t), \bar{u}(t)) = 0 ~ \text{and} ~ r(t)\leq 0 ~ \text{a.e.~in} ~ [t_{0}, t_{1}],\\
    &(p(t_{0}), -p(t_{1})) \in \lambda \partial \ell(\bar{x}(t_{0}), \bar{x}(t_{1})) + \mathcal{N}_{C}(\bar{x}(t_{0}), \bar{x}(t_{1})). 
\end{align*} 
\end{theorem}

The following is the definition of asymptotic weak maximum principle sequences.

\begin{definition}[AWMP Sequences]
\label{SeqAWMP}
    A sequence $\{(x^{\sind}, u^{\sind}, \zeta^{\sind}, p^{\sind}, q^{\sind}, r^{\sind}, \lambda_{\sind})\}\subset W_{n}^{1,1} \times  L^{1}_{m} \times L_{m}^{1}\times W_{n}^{1, 1} \times L_{m_{b}}^{1} \times L_{m_{g}}^{1} \times [0, +\infty)$ is called an \emph{Asymptotic Weak Maximum Principle sequence} (\emph{AWMP sequence}) if there exist $\delta > 0$ and sequences $\{(\sqa^{\sind}, \sqb^{\sind})\} \subset L_{n}^{1}\times L_{m}^{1}$, $\{\sqc^{\sind}_{j}\} \subset \mathscr{C} := \{\phi \in L_{1}^{1} : \phi(t) \leq 0\}$, $j=1, \dots, m_{g}$, and $\{(\sqd^{\sind}, \sqe^{\sind})\} \subset \mathbb{R}^{n}\times \mathbb{R}^{n}$ such that, for all $\sind\in \mathbb{N}$,
	\begin{enumerate}[label=(\roman*)]
	\item
	\label{iSeqAWMP}
        $\lambda_{\sind} + \Vert p^{\sind}\Vert_{L_{n}^{\infty}} = 1$;
		
		\item
		\label{iiSeqAWMP} $(-\dot{p}^{\sind}(t), \zeta^{\sind}(t)) - (\sqa^{\sind}(t), \sqb^{\sind}(t)) \in \co\partial_{x, u}\mathcal{H}(t, x^{\sind}(t), p^{\sind}(t), q^{\sind}(t), r^{\sind}(t), u^{\sind}(t))$ \\$ \text{a.e.~in} ~ [t_{0}, t_{1}]$;
		
		\item
		\label{iiiSeqAWMP}
		$\zeta^{\sind}(t) \in \mathcal{N}_{U_{\delta}(t)}(u^{\sind}(t)) ~ \text{a.e.~in} ~ [t_{0}, t_{1}]$;
		
		\item
		\label{ivSeqAWMP} $r_{j}^{\sind}(t) g_{j}^{-}(t, x^{\sind}(t), u^{\sind}(t)) = \sqc^{\sind}_{j}(t)$ and $r_{j}^{\sind}(t)\leq 0  ~ \text{a.e.~in} ~ [t_{0}, t_{1}], ~ j=1, \dots, m_{g}$;

		\item
		\label{vSeqAWMP} $(p^{\sind}(t_{0}), -p^{\sind}(t_{1})) -  (\sqd^{\sind}, \sqe^{\sind}) \in \lambda_{\sind}\partial \ell(x^{\sind}(t_{0}), x^{\sind}(t_{1})) + \mathcal{N}_{C}(x^{\sind}(t_{0}), x^{\sind}(t_{1}))$;
	\end{enumerate}
with $(\sqa^{\sind},\sqb^{\sind}) \rightharpoonup (0,0)$ in $L_{n}^{1} \times L_{m}^{1}$, $\sqc^{\sind}_{j}(t) \rightarrow 0$ almost everywhere in $[t_{0}, t_{1}]$, $j=1, \dots, m_{g}$, and $(\sqd^{\sind},\sqe^{\sind}) \rightarrow (0,0)$.
\end{definition}

For the purposes of this study, a minor adjustment has been made to the concept of AWMP sequences as introduced by Moreira and de Oliveira in \cite{Moreira2024art}. This refinement consists in changing the type of convergence of the sequences $(\sqa^{\sind},\sqb^{\sind})$ from pointwise to weak convergence. This modification renders the stationarity conditions, given by an Euler-Lagrange-type inclusion, more flexible, as demonstrated in Item \ref{iiSeqAWMP} of Definition \ref{SeqAWMP} above. It is important to mention that these modifications do not invalidate the results already proved in \cite{Moreira2024art}, since a few simple adjustments to their proofs will suffice.

It should be noted that, in Definition \ref{SeqAWMP}, no specific form of convergence is required for the sequence $\{(x^{\sind}, u^{\sind}, \zeta^{\sind}, p^{\sind}, q^{\sind}, r^{\sind}, \lambda_{\sind})\}$. By imposing a convergence criterion on the sequence $\{(x^{\sind}, u^{\sind})\}$ we arrive at the following definition.

\begin{definition}[Asymptotic Extremal] \label{Def_AWMP-Proces}
	A feasible process $(\bar{x}, \bar{u})$ of (P) is called an \emph{asymptotic extremal} if there exists an AWMP sequence $\{(x^{\sind} , u^{\sind}, \zeta^{\sind}, p^{\sind}, q^{\sind}, r^{\sind},$ $\lambda_{\sind})\}$ such that $x^{\sind} \rightarrow \bar{x}$ uniformly and $u^{\sind} \to \bar{u}$ in $L_{m}^{1}$.
\end{definition}

We make no specific assumptions about the convergence or boundedness of the sequence of multipliers $\{(q^{\sind}, r^{\sind})\}$. Furthermore, the uniform convergence of the sequence $\{x^{\sind}\}$ can be expressed as $\Vert x^{\sind} - \bar{x}\Vert_{L_{n}^{\infty }} \rightarrow 0$.

As previously stated, Theorem 3.3 from \cite{Moreira2024art} remains valid with the modification to the AWMP sequence concept, necessitating only adjustments to the portion of the proof that verifies the convergences, adapting them to the new type established in Definition \ref{SeqAWMP}. For the sake of clarity, the result is restated here, omitting the proof as it is quite lengthy; details can be found in \cite{Moreira2024art}.

\begin{theorem}[Asymptotic Weak Maximum Principle] \label{Teo_AWMP}
Let $(\bar{x}, \bar{u})$ be a weak local minimizer of (P). If, for some $\delta > 0$, Hypotheses (H1)--(H7) are satisfied, then $(\bar{x}, \bar{u})$ is an asymptotic extremal.
\end{theorem}

It is clear that the classical weak maximum principle implies the asymptotic weak maximum principle when considering constant sequences in Definition \ref{SeqAWMP}.

\section{A New Constraint Qualification for Mixed Constrained Optimal Control} \label{Sec_New_CQ}

In this section, guided by the notions explored in the context of nonlinear programming \cite{Andreani2011, Borgens2020, Mehlitz2020}, we investigate the conditions under which the necessary optimality conditions established in Theorem \ref{Teo_AWMP}, also referred to simply as AWMP conditions, imply the classical weak maximum principle (Theorem \ref{Teo_PMF_Versao_Class}). Specifically, in the finite-dimensional case for smooth problems, this involves the Cone Continuity Property (CCP) condition Andreani et al. \cite{Andreani2011}. In the finite-dimensional case for nonsmooth problems, the Asymptotically Mordukhovich-regularity (AM-regularity) condition Mehlitz \cite{Mehlitz2020}, and, in the infinite-dimensional case for smooth problems, the AKKT-regularity condition B\"{o}rgens et al. \cite{Borgens2020}. In accordance with the aforementioned inspiration, a novel regularity condition associated with the AWMP conditions is hereby proposed. This condition is referred to as AWMP-regularity. As will be demonstrated in this section, this new condition assumes the role of a constraint qualification, thereby ensuring the validity of the classical weak maximum principle.

In the consideration of the AWMP conditions, given $\delta > 0$ and $(\bar{x}, \bar{u}) \in W_{n}^{1,1} \times L_{m}^{1}$, for each 
$$
(x,u) \in \Omega_{\delta} := \{ (x,u) \in W_{n}^{1,1} \times L_{m}^{1} : (x(t),u(t)) \in \Omega_\delta(t) ~ \text{a.e. in} ~ [t_0,t_1] \},
$$ 
and $\sqc \in \mathscr{C}$, we consider the following two sets. 
\begin{definition}
$\mathcal{M}(x, u)$ is defined as the set of all $(\varphi, \psi, \gamma, \xi) \in L_{n}^{1}\times L_{m}^{1}\times \mathbb{R}^{n}\times \mathbb{R}^{n}$ such that
\begin{align*}
& (\varphi(t), \psi(t)) \in \co\partial_{x, u}(p(t)\cdot f(t, x(t), u(t))) -\{(-\dot{p}(t), 0)\} ~ \text{a.e. in} ~ [t_0,t_1], \\
& (\gamma, \xi) \in \lambda\partial\ell(x(t_{0}), x(t_{1})) - \{(p(t_{0}), -p(t_{1}))\}, \\
& (\lambda, p) \in [0, +\infty)\times W_{n}^{1,1} \text{ with } \lambda + \Vert p\Vert_{L_{n}^{\infty}} \neq 0.
\end{align*}
\end{definition}

\begin{definition}
$\mathcal{M}_{\delta}(x, u, \sqc)$, is defined as the set of all $(\varphi, \psi, \gamma, \xi)  \in L_{n}^{1}\times L_{m}^{1}\times \mathbb{R}^{n}\times \mathbb{R}^{n}$ satisfying
\begin{align*}
& (\varphi(t), \psi(t)) = \sum_{i=1}^{m_{b}}q_{i}(t)\nabla_{x, u}b_{i}(t, x(t), u(t)) + \sum_{j=1}^{m_{b}}r_{j}(t)\nabla_{x, u}g_{j}(t, x(t), u(t)) \\
& \hspace{2.1cm} + (0, -\zeta(t)) ~ \text{a.e. in} ~ [t_0,t_1], \\
& (q, r, \zeta) \in L_{m_{b}}^{1} \times L_{m_{g}}^{1} \times L_{m}^{1},  \quad \zeta(t) \in \mathcal{N}_{U_{\delta}(t)}(u(t)) ~ \text{a.e. in} ~ [t_0,t_1], \\
& r_{j}(t)g_{j}^{-}(t, x(t), u(t)) = \sqc_{j}(t), ~ r_{j}(t)  \leq 0 ~ \text{a.e. in} ~ [t_0,t_1], ~ j=1, \dots, m_{g}, \\
& |\psi(t)|\leq k_{\psi}(t) ~ \text{a.e. in} ~ [t_0,t_1] ~ \text{for some} ~ k_{\psi} \in L_{m}^{1} ~ \text{depending only on} ~ (\bar{x},\bar{u}) ~ \text{and} ~ \delta, \\
& (\gamma, \xi) \in \mathcal{N}_{C}(x(t_{0}), x(t_{1})).
\end{align*}
\end{definition}

The following lemma presents an alternative way to express the classical weak maximum principle (Theorem \ref{Teo_PMF_Versao_Class}).
\begin{lemma} 
\label{Lema_Rees_WMP}
    Let $(\bar{x}, \bar{u})$ be a feasible process for (P). If, for some $\delta > 0$, Hypotheses (H1)--(H6) are satisfied, then the conditions of the classical weak maximum principle (Theorem \ref{Teo_PMF_Versao_Class}) are satisfied at $(\bar{x}, \bar{u})$ if, and only if, 
    \begin{align*}
        \mathcal{M}(\bar{x}, \bar{u}) \cap \left(-\mathcal{M}_{\delta}(\bar{x}, \bar{u}, 0)\right) \neq \varnothing.
    \end{align*}
\end{lemma}
\begin{proof} 
Assume that the classical weak maximum principle (Theorem \ref{Teo_PMF_Versao_Class}) is valid at $(\bar{x}, \bar{u})$. Then, after a normalization of the multipliers, there exist $(\lambda, p, \zeta, q, r) \in [0, +\infty) \times W_{n}^{1,1} \times L_{m}^{1} \times L_{m_{b}}^{1} \times L_{m_{g}}^{1}$ such that
\begin{align}
& \lambda + \Vert p \Vert_{L_{n}^{\infty}} = 1, \label{Eq_Dr_1} \\
& (-\dot{p}(t), \zeta(t)) \in \co\partial_{x, u} \mathcal{H}(t, \bar{x}(t), p(t), q(t), r(t), \bar{u}(t)) ~ \text{a.e.~in} ~ [t_0,t_1], \label{Eq_Dr_2} \\
& \zeta(t) \in \co\mathcal{N}_{U(t)}(\bar{u}(t)) ~ \text{a.e.~in} ~ [t_0,t_1], \label{Eq_Dr_3} \\
& r(t) \cdot g(t, \bar{x}(t), \bar{u}(t)) = 0 \text{ and } r(t) \leq 0 ~ \text{a.e.~in} ~ [t_0,t_1], \label{Eq_Dr_4} \\
& (p(t_{0}), -p(t_{1})) \in \lambda \partial \ell(\bar{x}(t_{0}), \bar{x}(t_{1})) + \mathcal{N}_{C}(\bar{x}(t_{0}), \bar{x}(t_{1})). 
\label{Eq_Dr_5}
\end{align} 
Since, by hypothesis, $(\bar{x}, \bar{u})$ is feasible, from \eqref{Eq_Dr_4} we have
\begin{align*}
r_{j}(t) g_{j}^{-}(t, \bar{x}(t), \bar{u}(t)) & = r_{j}(t)(\max\{-g_{j}(t, \bar{x}(t), \bar{u}(t)), 0\}) \\ & = r_{j}(t)(-g_{j}(t, \bar{x}(t), \bar{u}(t))) = 0
\end{align*}
almost everywhere in $[t_{0}, t_{1}]$, for $j = 1, \dots, m_{g}$. Setting
\begin{align*}
    (\varphi(t), \psi(t)) &:= -\sum_{i=1}^{m_{b}} q_{i}(t) \nabla_{x, u} b_{i}(t, \bar{x}(t), \bar{u}(t)) - \sum_{j=1}^{m_{g}} r_{j}(t) \nabla_{x, u} g_{j}(t, \bar{x}(t), \bar{u}(t))\\
    &\quad~~ + (0, \zeta(t)) ~\text{a.e.~in} ~ t \in [t_{0}, t_{1}],
\end{align*} 
from \eqref{Eq_Dr_2}, we get 
\begin{align} \label{Eq_Rev_101}
(\varphi(t), \psi(t)) \in \co\partial_{x, u}(p(t) \cdot f(t, \bar{x}(t), \bar{u}(t))) - (-\dot{p}(t), 0) \text{ a.e. in } [t_{0}, t_{1}].
\end{align}
By Clarke \cite[Proposition 2.6.4]{Clarke1983}, 
\begin{align*}
    \co\partial_{x, u}(p(t) \cdot f(t, \bar{x}(t), \bar{u}(t))) = p(t)D_{x, u}f(t, \bar{x}(t), \bar{u}(t))
\end{align*}
almost everywhere in $[t_{0}, t_{1}]$, and by \cite[Proposition 2.6.2]{Clarke1983}, $D_{x,u} f(t, \bar{x}(t), \bar{u}(t)) \subset k_{f}(t)B$ almost everywhere in $[t_{0}, t_{1}]$, where $D_{x,u}f(t,x,u)$ denotes the Clarke's Generalized Jacobian\footnote{See Clarke \cite[Def. 2.6.1]{Clarke1983} for details.} of $f$ with respect to $(x, u)$ at the point $(t, x, u)$. Thus, since $\dot{p} \in L_{n}^{1}$, because $p \in W_{n}^{1,1}$, using \eqref{Eq_Rev_101}, we can conclude that $(\varphi, \psi) \in L_{n}^{1} \times L_{m}^{1}$. Additionally, from \eqref{Eq_Dr_1} and \eqref{Eq_Rev_101}, by defining $k_{\psi}(t) := k_{f}(t)$ almost everywhere in $[t_{0}, t_{1}]$, we conclude that 
$$
|\psi(t)| \leq k_{\psi}(t) ~ \text{a.e.~in} ~ [t_{0}, t_{1}].
$$
It follows directly from \eqref{Eq_Dr_3} and (H2) that 
$$
\zeta(t) \in \co\mathcal{N}_{U(t)}(\bar{u}(t)) = \mathcal{N}_{U_{\delta}(t)}(\bar{u}(t)) ~ \text{a.e.~in} ~ [t_{0}, t_{1}]
$$ 
for any $\delta > 0$. From \eqref{Eq_Dr_5}, we see that 
$$
(\gamma, \xi) \in \lambda \partial \ell(\bar{x}(t_{0}), \bar{x}(t_{1})) - \{(p(t_{0}), -p(t_{1}))\}
$$ 
for some $-(\gamma, \xi) \in \mathcal{N}_{C}(\bar{x}(t_{0}), \bar{x}(t_{1}))$. Therefore, using \eqref{Eq_Dr_1}, we conclude the first part of the result.

Conversely, let $(\varphi, \psi, \gamma, \xi) \in \mathcal{M}(\bar{x}, \bar{u}) \cap \left(-\mathcal{M}_{\delta}(\bar{x}, \bar{u}, 0)\right)$. By definition, \linebreak $(\varphi, \psi, \gamma, \xi) \in L_{n}^{1} \times L_{m}^{1} \times \mathbb{R}^{n} \times \mathbb{R}^{n}$ and there exists $(\lambda, p, q, r, \zeta) \in [0, +\infty) \times W_{n}^{1,1} \times L_{m}^{1} \times L_{m_{b}}^{1} \times L_{m_{g}}^{1} \times L_{m}^1$ such that
\begin{align}
& (\varphi(t), \psi(t)) \in \co \partial_{x, u}(p(t) \cdot f(t, \bar{x}(t), \bar{u}(t))) - \{(-\dot{p}(t), 0)\} ~ \text{a.e.~in} ~ [t_{0}, t_{1}], \label{Eq_Dr_6} \\
& 
\nonumber
(\varphi(t), \psi(t)) = -\sum_{i=1}^{m_{b}} q_{i}(t) \nabla_{x, u} b_{i}(t, \bar{x}(t), \bar{u}(t)) - \sum_{j=1}^{m_{g}} r_{j}(t) \nabla_{x, u} g_{j}(t, \bar{x}(t), \bar{u}(t)) \\
&\hspace{2.15cm}+ (0, \zeta(t)) ~ \text{a.e.~in} ~ [t_{0}, t_{1}], \label{Eq_Dr_7} \\
& \zeta(t) \in \mathcal{N}_{U_{\delta}}(\bar{u}(t)) ~ \text{a.e.~in} ~ [t_{0}, t_{1}], \label{Eq_Dr_8} \\
& r_{j}(t) g_{j}^{-}(t, \bar{x}(t), \bar{u}(t)) = 0 \text{ and } r_{j}(t) \leq 0 ~ \text{a.e.~in} ~ [t_{0}, t_{1}], ~ j = 1, \dots, m_{g}, \label{Eq_Dr_9} \\
& (\gamma, \xi) \in \lambda \partial \ell(\bar{x}(t_{0}), \bar{x}(t_{1})) - \{(p(t_{0}), -p(t_{1}))\}, \label{Eq_Dr_10} \\
& -(\gamma, \xi) \in \mathcal{N}_{C}(\bar{x}(t_{0}), \bar{x}(t_{1})), \label{Eq_Dr_11} \\
& \lambda + \Vert p \Vert_{L_{n}^{\infty}} \neq 0. \label{Eq_Dr_12}
\end{align}
From \eqref{Eq_Dr_6} and \eqref{Eq_Dr_7}, we obtain
\begin{align*}
(-\dot{p}(t), \zeta(t)) & \in \co\partial_{x, u}(p(t)\cdot f(t, \bar{x}(t), \bar{u}(t))) \\
&\quad~ + \sum_{i=1}^{m_{b}}q_{i}(t)\nabla_{x, u}b_{i}(t, \bar{x}(t), \bar{u}(t)) + \sum_{j=1}^{m_{g}}r_{j}(t)\nabla_{x, u}g_{j}(t, \bar{x}(t), \bar{u}(t)) \\
& = \co \partial_{x, u}\mathcal{H}(t, \bar{x}(t), p(t), q(t), r(t), \bar{u}(t)) ~ \text{a.e.~in} ~ [t_0,t_1],
\end{align*} 
where the last equality holds because of \cite[Corollary of Proposition 2.2.1 and Corollary 1 of Proposition 2.3.3]{Clarke1983}. Furthermore, from \eqref{Eq_Dr_8}, 
$$
\zeta(t) \in \mathcal{N}_{U_{\delta}}(\bar{u}(t)) = \co\mathcal{N}_{U}(\bar{u}(t)) ~ \text{a.e.~in} ~ [t_0,t_1],
$$
and from \eqref{Eq_Dr_10}, \eqref{Eq_Dr_11}, and \eqref{Eq_Dr_12},
\begin{align*}
(p(t_{0}), -p(t_{1})) & \in \lambda \partial \ell(\bar{x}(t_{0}), \bar{x}(t_{1})) - \{(\gamma, \xi)\} \\ & \subset \lambda \partial \ell(\bar{x}(t_{0}), \bar{x}(t_{1})) + \mathcal{N}_{C}(\bar{x}(t_{0}), \bar{x}(t_{1}))
\end{align*}
with $\lambda + \Vert p \Vert_{L_{n}^{\infty}} \neq 0$. Now, note that from the feasibility of $(\bar{x}, \bar{u})$ and from \eqref{Eq_Dr_9}, 
\begin{align*}
r_{j}(t)(-g_{j}(t, \bar{x}(t), \bar{u}(t))) &= r_{j}(t)(\max\{-g_{j}(t, \bar{x}(t), \bar{u}(t)), 0\}) \\
& = r_{j}(t)g_{j}^{-}(t, \bar{x}(t), \bar{u}(t)) = 0 ~ \text{a.e.~in} ~ [t_0,t_1], ~ j = 1, \dots, m_{g},
\end{align*}
which concludes the proof.
\end{proof}

Lemma \ref{Lema_Reesc_AWMP}, stated and proved below, ensures that, under certain hypotheses, we can represent the AWMP conditions in an alternative form involving the weak outer/upper limit of the Painlevé-Kuratowski type of the set-valued map $\mathcal{M}_{\delta}: W_{n}^{1,1}\times L_{m}^{1}\times \mathscr{C} \rightsquigarrow L_{n}^{1}\times L_{m}^{1}\times \mathbb{R}^{n}\times \mathbb{R}^{n}$ given by
\begin{align*}
	\wl\limsup_{\genfrac{}{}{0pt}{}{x \to \bar{x}, u \to \bar{u}}{\sqc \to 0}}\mathcal{M}_{\delta}(x, u, \sqc) &:= 
	\left\{ (\varphi, \psi, \gamma, \xi) \in L_{n}^{1}\times L_{m}^{1}\times \mathbb{R}^{n}\times \mathbb{R}^{n} : \right.\\
    &\left. \qquad
    \exists ~ \{(x^{\sind}, u^{\sind}, \sqc^{\sind})\} \subset W_{n}^{1,1}\times L_{m}^{1}\times \mathscr{C} ~ \text{such that} \right.\\
	&\left. \qquad x^{\sind} \to \bar{x} \text{ uniformly}, ~ u^{\sind} \to \bar{u} \text{ in } L_{m}^{1}, \right.\\
	&\left. \qquad \sqc_{j}^{\sind}(t) \to 0 ~ \text{ a.e. in } [t_{0}, t_{1}], ~ j = 1, \dots, m_{g}, \right.\\
	&\left. \qquad (\varphi^{\sind}, \psi^{\sind}) \rightharpoonup (\varphi, \psi) \text{ in } L_{n}^{1}\times L_{m}^{1}, \quad (\gamma^{\sind}, \xi^{\sind}) \to (\gamma, \xi), \right.\\
	&\left. \qquad (\varphi^{\sind}, \psi^{\sind}, \gamma^{\sind}, \xi^{\sind}) \in \mathcal{M}_{\delta}(x^{\sind}, u^{\sind}, \sqc^{\sind}) \text{ for all } \sind \in \mathbb{N} 
    \right\}.
\end{align*}
\begin{lemma} \label{Lema_Reesc_AWMP}
Let $(\bar{x}, \bar{u})$ be an asymptotic extremal process for (P). Suppose that Hypotheses (H1)-(H6) are satisfied.  Regarding the Definition \ref{SeqAWMP}, let the associated AWMP sequence be given by 
$$
\{(x^{\sind}, u^{\sind}, \zeta^{\sind}, p^{\sind}, q^{\sind}, r^{\sind}, \lambda_{\sind})\} \subset W_{n}^{1,1} \times L_{m}^{1} \times L_{m}^{1} \times W_{n}^{1,1} \times L_{m_{b}}^{1} \times L_{m_{g}}^{1} \times [0, +\infty), 
$$
with $\{(\sqa^{\sind}, \sqb^{\sind})\} \subset L_{n}^{1} \times L_{m}^{1}$. Assume that
\begin{itemize}
\item[(H8)] there exist integrable functions $c_{\sqa}$ and $c_{\sqb}$,  depending only on $(\bar{x},\bar{u})$ and $\delta$, such that 
$$
|\sqa^{\sind}(t)| \leq c_{\sqa}(t) ~ \text{and} ~ |\sqb^{\sind}(t)| \leq c_{\sqb}(t) ~ \text{a.e.~in} ~ [t_{0}, t_{1}] ~ \forall \, \sind \in \mathbb{N};
$$
\item[(H9)] there exists an integrable function $c_{\varphi}$, depending only on $(\bar{x},\bar{u})$ and $\delta$, such that for all $\sind \in \mathbb{N}$ and almost everywhere in $[t_{0}, t_{1}]$,
\begin{align*}
\left\vert \sum_{i=1}^{m_{b}}q_{i}^{\sind}(t)\nabla_{x}b_{i}(t, x^{\sind}(t), u^{\sind}(t)) + \sum_{j=1}^{m_{g}}r_{j}^{\sind}(t)\nabla_{x}g_{j}(t, x^{\sind}(t), u^{\sind}(t))\right\vert \leq c_{\varphi}(t).
\end{align*}
\end{itemize}
Then
\begin{align*}
\mathcal{M}(\bar{x}, \bar{u})\cap \left(- \wl\limsup_{\genfrac{}{}{0pt}{}{x \to \bar{x}, u \to \bar{u}}{\sqc \to 0}}\mathcal{M}_{\delta}(x, u, \sqc)\right) \neq \varnothing.
\end{align*}
\end{lemma}
\begin{proof}
As, by hypothesis, $(\bar{x}, \bar{u})$ is an asymptotic extremal process for (P), then, by definition, there exist $\delta > 0$ and sequences $\{(x^{\sind}, u^{\sind}, \zeta^{\sind}, p^{\sind}, q^{\sind}, r^{\sind}, \lambda_{\sind})\} \subset W_{n}^{1,1} \times L_{m}^{1} \times L_{m}^{1} \times W_{n}^{1,1} \times L_{m_{b}}^{1} \times L_{m_{g}}^{1} \times [0, +\infty)$, $\{(\sqa^{\sind}, \sqb^{\sind})\} \subset L_{n}^{1} \times L_{m}^{1}$, $\{\sqc^{\sind}\} \subset \mathscr{C}$ and $\{(\sqd^{\sind}, \sqe^{\sind})\} \subset \mathbb{R}^{n} \times \mathbb{R}^{n}$ such that, for each $\sind \in \mathbb{N}$,
	\begin{align}
		&\lambda_{\sind} + \Vert p^{\sind}\Vert_{L_{n}^{\infty}} = 1,
	\label{Eq_Dr_13}\\
    \nonumber
		&(-\dot{p}^{\sind}(t), \zeta^{\sind}(t)) - (\sqa^{\sind}(t), \sqb^{\sind}(t))  \\
        &\hspace{2.2cm} \in \co\partial_{x, u}\mathcal{H}(t, x^{\sind}(t), p^{\sind}(t), q^{\sind}(t), r^{\sind}(t), u^{\sind}(t)) ~ \text{a.e.~in} ~ [t_{0}, t_{1}],
	\label{Eq_Dr_14}\\
		&\zeta^{\sind}(t) \in \mathcal{N}_{U_{\delta}(t)}(u^{\sind}(t)) ~ \text{a.e.~in} ~ [t_{0}, t_{1}],
	\label{Eq_Dr_15}\\
		&r_{j}^{\sind}(t) g_{j}^{-}(t, x^{\sind}(t), u^{\sind}(t)) = \sqc^{\sind}_{j}(t) ~ \text{and} ~ r_{j}^{\sind}(t) \leq 0 ~ \text{a.e.~in} ~ [t_{0}, t_{1}], ~ j=1, \dots, m_{g}, 
	\label{Eq_Dr_16}\\
		&(p^{\sind}(t_{0}), -p^{\sind}(t_{1})) - (\sqd^{\sind}, \sqe^{\sind}) \in \lambda_{\sind}\partial \ell(x^{\sind}(t_{0}), x^{\sind}(t_{1})) + \mathcal{N}_{C}(x^{\sind}(t_{0}), x^{\sind}(t_{1})),
	\label{Eq_Dr_17}
	\end{align}
where $(\sqa^{\sind},\sqb^{\sind}) \rightharpoonup (0,0)$ in $L_{n}^{1} \times L_{m}^{1}$, $\sqc^{\sind}_{j}(t) \to 0$ almost everywhere in $[t_{0}, t_{1}]$ for $j=1, \dots, m_{g}$, $(\sqd^{\sind},\sqe^{\sind}) \to (0,0)$, $x^{\sind} \to \bar{x}$ uniformly, and $u^{\sind} \to \bar{u}$ in $L_{m}^{1}$. 
	
It is implicit in \eqref{Eq_Dr_15} that $u^{\sind}(t) \in U_{\delta}(t)$ for each $\sind \in \mathbb{N}$ almost everywhere in $[t_{0}, t_{1}]$, as otherwise, by the definition of the normal cone, $\mathcal{N}_{U_{\delta}(t)}(u^{\sind}(t))$ would be empty.
	
By \eqref{Eq_Dr_17}, for each $\sind \in \mathbb{N}$, there exists 
\begin{align} \label{Eq_Revv_2}
(\pi_{1}^{\sind}, \pi_{2}^{\sind}) \in \mathcal{N}_{C}(x^{\sind}(t_{0}), x^{\sind}(t_{1}))
\end{align}
such that 
\begin{align}
(p^{\sind}(t_{0}), -p^{\sind}(t_{1})) - (\sqd^{\sind}, \sqe^{\sind}) \in \lambda_{\sind}\partial \ell(x^{\sind}(t_{0}), x^{\sind}(t_{1})) + \{(\pi_{1}^{\sind}, \pi_{2}^{\sind})\}.
\label{Eq_Revv_1}
\end{align}

Let us define
\begin{align} 
\nonumber
& (\varphi^{\sind}(t), \psi^{\sind}(t), \gamma^{\sind}, \xi^{\sind}) :=  \\
& \quad \left(-\sum_{i=1}^{m_{b}}q_{i}^{\sind}(t)\nabla_{x, u}b_{i}(t, x^{\sind}(t), u^{\sind}(t)) - \sum_{j=1}^{m_{g}}r_{j}^{\sind}(t)\nabla_{x, u}g_{j}(t, x^{\sind}(t), u^{\sind}(t)), 0, 0\right) \nonumber \\
& \qquad ~ + (0, \zeta^{\sind}(t), -\pi_{1}^{\sind}, -\pi_{2}^{\sind}) ~ \text{a.e.~in} ~ [t_{0}, t_{1}] ~ \forall \, \sind \in \mathbb{N}. \label{Def}
\end{align}

By \eqref{Eq_Dr_13}, $\lambda_{\sind} \leq 1$ and $|p^{\sind}(t_{1})| \leq 1$ for each $\sind \in \mathbb{N}$. Thus, we can extract convergent subsequences, say $\lambda_{\sind} \to \lambda$ and $p^{\sind}(t_{1}) \to p_{1}$ (again, we do not relabel), with $(\lambda, p_{1}) \in [0, +\infty) \times \mathbb{R}^{n}$.

From \eqref{Eq_Dr_14} and \eqref{Eq_Revv_1}, for each $\sind \in \mathbb{N}$,
\begin{align}
\nonumber
(\varphi^{\sind}(t), \psi^{\sind}(t)) - (\sqa^{\sind}(t), \sqb^{\sind}(t)) &\in \co\partial_{x, u}(p^{\sind}(t)\cdot f(t, x^{\sind}(t), u^{\sind}(t))) \\
&\quad - \{(-\dot{p}^{\sind}(t), 0)\} ~ \text{a.e.~in} ~ [t_{0}, t_{1}]
\label{Eq_Dr_18}
\end{align}
and
\begin{align}
(\gamma^{\sind}, \xi^{\sind}) - (\sqd^{\sind}, \sqe^{\sind}) \in \lambda_{\sind}\partial\ell(x^{\sind}(t_{0}), x^{\sind}(t_{1})) - \{(p^{\sind}(t_{0}), -p^{\sind}(t_{1}))\}.
\label{Eq_Dr_19}
\end{align}

As 
\begin{align*}
    \co\partial_{x, u}(p^{\sind}(t)\cdot f(t, x^{\sind}(t), u^{\sind}(t))) \subset p^{\sind}(t) D_{x, u}f(t, x^{\sind}(t), u^{\sind}(t))
\end{align*}
and $D_{x, u}f(t, x^{\sind}(t), u^{\sind}(t)) \subset k_{f}(t)B $, from \eqref{Eq_Dr_13} and \eqref{Eq_Dr_18} and Hypotheses (H8) and (H9), for each $\sind \in \mathbb{N}$,
\begin{align*}
|\dot{p}^{\sind}(t)| &\leq |p^{\sind}(t)|k_{f}(t) + \vert \sqa^{\sind}(t) \vert + |\varphi^{\sind}(t)| \\
&\leq |p^{\sind}(t)|k_{f}(t) + c_{\sqa}(t) + c_{\varphi}(t) \\
& \leq k_{f}(t) + c_{\sqa}(t) + c_{\varphi}(t) ~ \text{a.e.~in} ~ [t_{0}, t_{1}].
\end{align*}
Consequently, $\{\dot{p}^{\sind}\}$ is uniformly integrably bounded. By the Dunford-Pettis Theorem (see, for example, Vinter \cite[Theorem 2.5.1]{Vinter2000}), through the extraction of a subsequence (we do not relabel), there exists $p \in W_{n}^{1,1}$ with $p^{\sind} \to p$ uniformly and $\dot{p}^{\sind} \rightharpoonup \dot{p}$ in $L_{n}^{1}$.

It follows from Hypothesis (H9) that $\{\varphi^\sind\}$ is uniformly integrably bounded. By the Dunford-Pettis Theorem, extracting a subsequence (we do not relabel), we have $\varphi^{\sind} \rightharpoonup \varphi$ for some function $\varphi \in L_{n}^{1}$. 

From (H8), \eqref{Eq_Dr_13} and \eqref{Eq_Dr_18}, we have
\begin{align*}
|\psi^{\sind}(t)| \leq |p^{\sind}(t)|k_{f}(t) + |\sqb^{\sind}(t)|  \leq k_{f}(t) + c_{\sqb}(t) ~ \text{a.e.~in} ~ [t_{0}, t_{1}] ~ \forall \, \sind \in \mathbb{N}.
\end{align*}
Defining $k_{\psi}(t) := k_{f}(t) + c_{\sqb}(t) $ almost everywhere in $[t_{0}, t_{1}]$, we obtain 
\begin{align} \label{Cond_indep}
|\psi^{\sind}(t)| \leq k_{\psi}(t) ~ \text{a.e. in} ~ [t_{0}, t_{1}] ~ \forall\, \sind \in \mathbb{N}.
\end{align}
Consequently, $\{\psi^{\sind}\}$ is uniformly integrably bounded. Thus, by the Dunford-Pettis Theorem, by extracting a subsequence (we do not relabel), we conclude that $\psi^{\sind} \rightharpoonup \psi$ for some function $\psi \in L_{m}^{1}$.

By using \eqref{Eq_Dr_15}, \eqref{Eq_Dr_16}, \eqref{Eq_Revv_2}, \eqref{Def} and \eqref{Cond_indep}, we can conclude that
\begin{align*}
(\varphi^{\sind}, \psi^{\sind}, \gamma^{\sind}, \xi^{\sind}) \in -\mathcal{M}_{\delta}(x^{\sind}, u^{\sind}, \sqc^{\sind}) ~ \forall \, \sind \in \mathbb{N}.
\end{align*}

By (H3), we know that $\ell$ is locally Lipschitz. Then, from \eqref{Eq_Dr_13} and \eqref{Eq_Dr_19} and the fact that $(\sqd^{\sind},\sqe^{\sind}) \to (0,0)$, it follows that the sequence $\{(\gamma^{\sind}, \xi^{\sind})\}$ is bounded. Then, we can take a subsequence (again, we do not relabel) such that $(\gamma^{\sind}, \xi^{\sind}) \to (\gamma, \xi)$ for some $(\gamma, \xi) \in \mathbb{R}^{n}\times\mathbb{R}^{n}$.

Therefore,
\begin{align*}
(\varphi, \psi, \gamma, \xi) \in - \wl\limsup_{\genfrac{}{}{0pt}{}{x \to \bar{x}, u \to \bar{u}}{\sqc \to 0}}\mathcal{M}_{\delta}(x, u, \sqc).
\end{align*}

A direct modification in the proof of the Trajectory Compactness Theorem (see Clarke \cite[Theorem 3.1.7]{Clarke1983}) and an appeal to the upper semicontinuity properties of the normal cones and limit subdifferentials allow us to take limits in relations \eqref{Eq_Dr_18} and \eqref{Eq_Dr_19} to obtain
\begin{align*}
(\varphi(t), \psi(t)) \in \co \partial_{x, u}(p(t)\cdot f(t, \bar{x}(t), \bar{u}(t))) - \{(-\dot{p}(t), 0)\} ~ \text{a.e. in} ~ [t_0,t_1]
\end{align*}
and
\begin{align*}
(\gamma, \xi) \in \lambda\partial\ell(\bar{x}(t_{0}), \bar{x}(t_{1})) - \{(p(t_{0}), -p(t_{1}))\}.
\end{align*}
Furthermore, since $\lambda_{\sind} \to \lambda$ and $p^{\sind} \to p$ uniformly, we conclude from \eqref{Eq_Dr_13} that $\lambda + \Vert p\Vert_{L_{n}^{\infty}} \neq 0$. Therefore,
\begin{align*}
(\varphi, \psi, \gamma, \xi) \in \mathcal{M}(\bar{x}, \bar{u})
\end{align*}
and we conclude the proof.
\end{proof}

Inspired by the results provided by Lemmas \ref{Lema_Rees_WMP} and \ref{Lema_Reesc_AWMP}, we proceed to define AWMP-regularity for (P) as follows.
\begin{definition} 
\label{Def_AWMP-Regular}
	A feasible process $(\bar{x}, \bar{u})$ for Problem (P) is said to be \emph{Asymptotic Weak Maximum Principle regular} (or \emph{AWMP-regular}), if
	\begin{align*}
		\wl\limsup_{\genfrac{}{}{0pt}{}{x \to \bar{x}, u \to \bar{u}}{\sqc \to 0}}\mathcal{M}_{\delta}(x, u, \sqc) \subset \mathcal{M}_{\delta}(\bar{x}, \bar{u}, 0)
	\end{align*}
for some $\delta > 0$.
\end{definition}

In accordance with Definition \ref{Def_AWMP-Regular}, the subsequent results are immediately derivable from Theorem \ref{Teo_AWMP} and Lemmas \ref{Lema_Rees_WMP} and \ref{Lema_Reesc_AWMP}. In essence, these theorems assert two key points: (i) under the condition of AWMP-regularity, the asymptotic version of the weak maximum principle implies the classical one; and (ii) the AWMP-regularity condition serves as a constraint qualification for Problem (P), ensuring the validity of the classical weak maximum principle.

\begin{theorem} \label{AWMP_WMP}
Let $(\bar{x}, \bar{u})$ be an asymptotic extremal of (P). Assume that Hypotheses (H1)-(H9) are satisfied. If $(\bar{x}, \bar{u})$ is an AWMP-regular process of (P), then $(\bar{x}, \bar{u})$ satisfies the conditions of the classical weak maximum principle (Theorem \ref{Teo_PMF_Versao_Class}).
\end{theorem}

\begin{theorem}
\label{Teo_Nova_CQ}
Let $(\bar{x}, \bar{u})$ be a local weak minimizer for (P). Assume that Hypotheses (H1)-(H9) are satisfied. If $(\bar{x}, \bar{u})$ is an AWMP-regular process of (P), then $(\bar{x}, \bar{u})$ satisfies the conditions of the classical weak maximum principle (Theorem \ref{Teo_PMF_Versao_Class}).
\end{theorem}

\begin{example}
    In \cite[Example 3.5]{Moreira2024art}, Moreira and de Oliveira explore the fact that the classical weak maximum principle is not satisfied, while the asymptotic weak maximum principle holds. The problem in question is linear and is given by:
	\begin{align*}
		\begin{array}{ll}
			&\text{minimize } \ x(1) \\
			&\text{over } \ x\in W^{1, 1}([0, 1]; \mathbb{R}) \text{ and } u\in L^{1}([0, 1] ; \mathbb{R}^{2}) \text{ satisfying}\\
			&\hspace{0.7 cm}
			\begin{array}{ll}
				\dot{x}(t) = u_{1}(t) \text{ a.e. in } [0, 1], \\
				x(t) + u_{1}(t) + 2u_{2}(t) = 0 \text{ a.e. in } [0, 1], \\
				u_{1}(t) + 2u_{2}(t) = 0 \text{ a.e. in }[0, 1].
			\end{array}
		\end{array}
	\end{align*}
	The only feasible process is $(\bar{x}, (\bar{u}_{1}, \bar{u}_{2})) = (0, (0, 0))$, so that it is optimal. Moreover, due to the linearity of the problem, Hypotheses (H1)-(H7) are satisfied. 
 
    The AWMP sequence considered by Moreira and de Oliveira was given by
    \begin{align*}
        &x^{\sind}(t) = 0, ~ (u_{1}^{\sind}(t), u_{2}^{\sind}(t)) = (0, 0), ~ (\zeta_{1}^{\sind}(t), \zeta_{2}^{\sind}(t)) = (0, 0), ~ \lambda_{\sind} = \frac{1}{2},\\
        &p^{\sind}(t) = \begin{cases}
            0, ~ t \in \left[0, \frac{\sind}{\sind + 1}\right),\\
            -\left(\frac{\sind + 1}{2}\right)t + \frac{\sind}{2}, ~ t \in \left[\frac{\sind}{\sind + 1}, 1\right]
        \end{cases}\hspace{-0.35 cm}, ~(q_{1}^{\sind}(t), q_{2}^{\sind}(t)) = \left(-\dot{p}^{\sind}(t), \dot{p}^{\sind}(t)\right),\\
        &(\sqa^{\sind}(t), (\sqb_{1}^{\sind}(t), \sqb_{2}^{\sind}(t))) = \left(0, \left(-p^{\sind}(t), 0\right)\right) ~\text{and}~ (\sqd^{\sind}, \sqe^{\sind}) = \left(0, 0\right)
    \end{align*}
    almost everywhere in $[0, 1]$ for all $\sind \in \mathbb{N}$. Hypothesis (H8) is clearly satisfied. However, Hypothesis (H9) does not hold. Indeed, note that  
	\begin{align*}
		q_{1}^{\sind}(t)\nabla_{x}b_{1}(t, x^{\sind}(t), (u_{1}^{\sind}(t), u_{2}^{\sind}(t))) + q_{2}^{\sind}(t)\nabla_{x}b_{2}(t, x^{\sind}(t), (u_{1}^{\sind}(t), u_{2}^{\sind}(t))) = q_{1}^{\sind}(t)
	\end{align*} 
	almost everywhere in $[0, 1]$ for each $\sind \in \mathbb{N}$. Since  
	\begin{align*}
		q_{1}^{\sind}(t) = -\dot{p}^{\sind}(t) = 
		\begin{cases}
			0, ~ \text{if } t \in \left[0, \frac{\sind}{\sind + 1}\right),\\
			-\frac{\sind + 1}{2}, ~ \text{if } t \in \left(\frac{\sind}{\sind + 1}, 1\right]
		\end{cases}
	\end{align*}
	almost everywhere in $[0, 1]$ for all $\sind \in \mathbb{N}$, the sequence $\{q_{1}^{\sind}\}$ cannot be bounded almost everywhere in $[0, 1]$ by an integrable function independent of $\sind \in \mathbb{N}$. Consequently, Hypothesis (H8) is not satisfied.

As shown in Andreani et al. \cite[Example 4.1]{Andreani2020}, the conditions of the classical weak maximum principle do not hold for this problem. Therefore, it was expected that some of the hypotheses of Theorem \ref{Teo_Nova_CQ} would fail.
\end{example}

Subsequently, it will be demonstrated that the AWMP-regularity condition is the weakest condition under which the asymptotic weak maximum principle implies the classical weak maximum principle. The meaning of ``weakest condition'' should be understood in the sense of Theorem  \ref{Lem_Cond_mais_fraca}, presented below. Additionally, it is important to highlight that the analysis is limited to the smooth case. The nonsmooth case poses substantial technical challenges and, consequently, will be explored in future research.

Prior to the presentation of the statement of Theorem \ref{Lem_Cond_mais_fraca}, it is necessary to consider the following relation:
\begin{itemize} 
    \item[(R)] $(\bar{x}, \bar{u})$ is an asymptotic extremal for (P) with the sequence ${(\lambda_{\sind}, p^{\sind})} \subset [0, +\infty)$ $\times W_{n}^{1,1}$ obeying $\lambda_{\sind} \to \lambda$ and $p^{\sind} \to p$ uniformly $\Rightarrow$ $(\bar{x}, \bar{u})$ is a extremal process with the multipliers $(\lambda, p) \in [0, +\infty) \times W_{n}^{1,1}$. 
\end{itemize}

\begin{theorem} \label{Lem_Cond_mais_fraca}
Let $(\bar{x},\bar{u}) \in W_{n}^{1,1} \times L^1_m$ in which (H2)--(H6) are valid for some $\delta > 0$. Assume that 
\begin{align*}
& 0 = b(t,\bar{x}(t),\bar{u}(t)) ~ \text{a.e. in} ~ [t_0,t_1], \\
& 0 \geq g(t,\bar{x}(t),\bar{u}(t)) ~ \text{a.e. in} ~ [t_0,t_1], \\ 
& \bar{u}(t) \in  U(t) ~ \text{a.e. in} ~ [t_0,t_1], \\ 
& (\bar{x}(S),\bar{x}(T)) \in C. 
\end{align*}
If for each pair of continuously differentiable functions $\ell(\cdot,\cdot)$ and $f(t,\cdot,\cdot)$ the relation (R) is verified, then $(\bar{x},\bar{u})$ is an AWMP-regular process.
\end{theorem}
\begin{proof} 
	Let $\delta > 0$ be such that Hypotheses (H2)--(H6) are satisfied. Consider
	\begin{align*}
		(\varphi, \psi, \gamma, \xi) \in \wl\limsup_{\genfrac{}{}{0pt}{}{x \to \bar{x}, u \to  \bar{u}}{\sqc \to 0}}\mathcal{M}_{\delta}(x, u, \sqc).
	\end{align*}
	By definition, there exist $\{(x^{\sind}, u^{\sind}, \sqc^{\sind})\} \subset W_{n}^{1,1} \times L_{m}^{1} \times \mathscr{C}$ and $\{(\varphi^{\sind}, \psi^{\sind}, \gamma^{\sind}, \xi^{\sind})\} \subset L_{n}^{1} \times L_{m}^{1} \times \mathbb{R}^{n} \times \mathbb{R}^{n}$ such that $x^{\sind} \to \bar{x}$ uniformly, $u \to \bar{u}$ in $L_{m}^{1}$, $\sqc_{j}^{\sind}(t) \to 0$ almost everywhere in $[t_{0}, t_{1}]$ for $j = 1, \dots, m_{g}$, $(\varphi^{\sind}, \psi^{\sind}) \rightharpoonup (\varphi, \psi)$, $(\gamma^{\sind}, \xi^{\sind}) \to (\gamma, \xi)$, and
	\begin{align}
		(\varphi^{\sind}, \psi^{\sind}, \gamma^{\sind}, \xi^{\sind}) \in \mathcal{M}_{\delta}(x^{\sind}, u^{\sind}, \sqc^{\sind}) ~  \forall \, \sind \in \mathbb{N}.
	\label{Eq_Dr_21}
	\end{align}
From \eqref{Eq_Dr_21}, we guarantee the existence of $k_{\psi} \in L_{m}^{1}$ and sequences $\{(q^{\sind}, r^{\sind})\} \subset L_{m_{b}}^{1} \times L_{m_{g}}^{1}$ and $\{\zeta^{\sind}\} \subset L_{m}^{1}$ such that, 
\begin{align*}
& (\varphi^{\sind}(t), \psi^{\sind}(t)) = \sum_{i=1}^{m_{b}}q_{i}^{\sind}(t)\nabla_{x, u}b_{i}(t, x^{\sind}(t), u^{\sind}(t)) + \sum_{j=1}^{m_{g}}r_{j}^{\sind}(t)\nabla_{x, u}g_{j}(t, x^{\sind}(t), u^{\sind}(t)) \\
&\hspace{2.4cm}+ (0, -\zeta^{\sind}(t)) \\
& \zeta^{\sind}(t) \in \mathcal{N}_{U_{\delta}(t)}(u^{\sind}(t)), \\
& r_{j}^{\sind}(t)g_{j}^{-}(t, x^{\sind}(t), u^{\sind}(t)) = \sqc_{j}^{\sind}(t), ~ r_{j}^{\sind}(t) \leq 0, \\
& \vert\psi^{\sind}(t)\vert \leq k_{\psi}(t)
\end{align*}
almost everywhere in $[t_{0}, t_{1}]$ for all $\sind \in \mathbb{N}$. Moreover, we know that $\{(\gamma^{\sind}, \xi^{\sind})\} \subset \mathcal{N}_{C}(x^{\sind}(t_{0}), x^{\sind}(t_{1}))$.
	
By the definitions of $\mathcal{M}_{\delta}$ and the limiting normal cone, we have $u^{\sind}(t) \in U_{\delta}(t)$ almost everywhere in $[t_{0}, t_{1}]$ and $(x^{\sind}(t_{0}), x^{\sind}(t_{1})) \in C$ for each $\sind \in \mathbb{N}$, because otherwise $\mathcal{M}_{\delta}(x^{\sind}, u^{\sind}, \sqc^{\sind}) = \varnothing$ for each $\sind \in \mathbb{N}$, which would contradict \eqref{Eq_Dr_21}.
	
We define, for almost every $t \in [t_{0}, t_{1}]$, the following matrices:
	\begin{align*}
		A(t) := 
		\left[
		\begin{array}{cccc}
			\varphi_{1}(t) & 0 & \cdots & 0\\
			0 & \varphi_{2}(t) & \cdots & 0\\
			\vdots & \vdots & \ddots & \vdots\\
			0 & 0 & \cdots & \varphi_{n}(t)
		\end{array}
		\right]_{n \times n}
	\end{align*}
    and 
    \begin{align*}
        B(t) := 
		\left[
		\begin{array}{cccc}
			\psi_{1}(t) & \psi_{2}(t) & \cdots & \psi_{m}(t)\\
			0 & 0 & \cdots & 0\\
			\vdots & \vdots & \ddots & \vdots\\
			0 & 0 & \cdots & 0
		\end{array}
		\right]_{n \times m}.
    \end{align*}
We also consider the following auxiliary optimal control problem:
$$
\begin{array}{rl}
\text{minimize} & \ell(x(t_{0}), x(t_{1})) \\
\text{over} & (x, u)\in W_{n}^{1, 1}\times L_{m}^{1} \text{ satisfying} \\
& \dot{x}(t) = f(t, x(t), u(t)) ~ \text{a.e. in} ~ [t_{0}, t_{1}], \\
& 0 = b(t, x(t), u(t)) ~ \text{a.e. in} ~ [t_{0}, t_{1}], \\
& 0 \geq g(t, x(t), u(t)) ~ \text{a.e. in} ~ [t_{0}, t_{1}], \\
& u(t) \in U(t) ~ \text{a.e. in} ~ [t_{0}, t_{1}], \\
& (x(t_{0}), x(t_{1})) \in C,
\end{array} 
\eqno{\text{(ACP)}}
$$
where
\begin{align}
& \ell(x,y) := (-2\gamma + \mathbf{1}_{n\times 1}) \cdot x + (-2\xi - \mathbf{1}_{n\times 1}) \cdot y,
\label{Eq_Dr_24} \\
& f(t, x, u) := -2A(t)x - 2B(t)u ~ \text{a.e. in} ~ [t_{0}, t_{1}]. \label{Eq_Dr_25}
\end{align}

For all $\sind \in \mathbb{N}$ and $t \in [t_{0}, t_{1}]$, we set
\begin{align*}
    \lambda_{\sind} := \dfrac{1}{2} \quad\text{ and }\quad p^{\sind}(t) := \left(\mathbf{\dfrac{1}{2}}\right)_{n\times 1}.
\end{align*}
Thus, $(\lambda_{\sind}, p^{\sind}) \in [0, +\infty)\times W_{n}^{1,1}$, $\lambda_{\sind} + \Vert p^{\sind}\Vert_{L_{n}^{\infty}} = 1$ and $\dot{p}^{\sind}(t) \equiv 0$ for each $\sind \in \mathbb{N}$. Moreover, $\lambda_{\sind} \to 1/2 =: \lambda$ and $(p^{\sind}, \dot{p}^{\sind}) \to (p, \dot{p}) = (1/2, 0)$ uniformly. For every $\sind \in \mathbb{N}$ and almost every $t \in [t_{0}, t_{1}]$, we also set
\begin{align}
    (\sqa^{\sind}(t), \sqb^{\sind}(t)) &:= -\nabla_{x, u}f(t, x^{\sind}(t), u^{\sind}(t))^{T}p^{\sind}(t) + (-\dot{p}^{\sind}(t), 0) - (\varphi^{\sind}(t), \psi^{\sind}(t)) 
\label{Eq_Dr_22}
\end{align}
and
\begin{align}
    (\sqd^{\sind}, \sqe^{\sind}) &:= -\lambda_{\sind}\nabla\ell(x^{\sind}(t_{0}), x^{\sind}(t_{1})) + (p^{\sind}(t_0), -p^{\sind}(t_1)) - (\gamma^{\sind}, \xi^{\sind}).
\label{Eq_Dr_23}
\end{align}
From \eqref{Eq_Dr_24}--\eqref{Eq_Dr_23}, $(\sqa^{\sind}(t), \sqb^{\sind}(t)) = (\varphi(t), \psi(t)) - (\varphi^{\sind}(t), \psi^{\sind}(t))$ almost everywhere in $[t_{0}, t_{1}]$ and $(\sqd^{\sind}, \sqe^{\sind}) = (\gamma, \xi) - (\gamma^{\sind}, \xi^{\sind})$ for each $\sind \in \mathbb{N}$. Hence, $(\sqa^{\sind},\sqb^{\sind}) \rightharpoonup (0,0)$ in $L_{n}^{1} \times L_{m}^{1}$ and $(\sqd^{\sind},\sqe^{\sind}) \to (0,0)$. Consequently, $(\bar{x}, \bar{u})$ is an asymptotic extremal for (ACP). Therefore, by hypothesis, $(\bar{x}, \bar{u})$ is an extremal process of (ACP) with the multipliers $(\lambda, p(t)) \equiv (1/2, 1/2)$, that is, there exist $\zeta \in L_{m}^{1}$ and $(q, r) \in L_{m_{b}}^{1}\times L_{m_{g}}^{1}$ such that
\begin{align}
\nonumber
    &\lambda + \Vert p\Vert_{L_{n}^{\infty}} \neq 0,\\
    &(-\dot{p} (t), \zeta(t)) = \nabla_{x, u} \mathcal{H}(t, \bar{x}(t), p(t), q(t), r(t), \bar{u}(t)) ~ \text{a.e. in} ~ [t_0,t_1],
\label{Eq_Dr_26}\\
    &\zeta(t) \in \co\mathcal{N}_{U(t)}(\bar{u}(t)) ~ \text{a.e. in} ~ [t_0,t_1],
\label{Eq_Dr_30}\\
\nonumber
    &r(t)\cdot g(t, \bar{x}(t), \bar{u}(t)) = 0 ~ \text{and} ~ r(t) \leq 0 ~ \text{a.e. in} ~ [t_0,t_1], \\
    &(p(t_{0}), -p(t_{1})) \in \lambda \nabla \ell(\bar{x}(t_{0}), \bar{x}(t_{1})) + \mathcal{N}_{C}(\bar{x}(t_{0}), \bar{x}(t_{1})).
\label{Eq_Dr_27}
\end{align}
Since, by assumption, $0 \geq g(t,\bar{x}(t),\bar{u}(t))$ almost everywhere in $[t_0,t_1]$, 
\begin{align*}
r_{j}(t)g_{j}^{-}(t, \bar{x}(t), \bar{u}(t)) = r_{j}(t)(-g_{j}(t, \bar{x}(t), \bar{u}(t))) = 0 ~ \text{a.e. in} ~ [t_0,t_1], ~ j = 1, \dots, m_{g}.
\end{align*}
By \eqref{Eq_Dr_26}, for almost everywhere in $[t_0,t_1]$,
\begin{multline} \label{Eq_Dr_28}
- \nabla_{x, u}f(t, \bar{x}(t), \bar{u}(t))^{T}p(t) + (-\dot{p}(t), 0) \\ = \sum_{i=1}^{m_{b}}q_{i}(t)\nabla_{x, u}b_{i}(t, \bar{x}(t), \bar{u}(t)) + \sum_{j=1}^{m_{g}}r_{j}(t)\nabla_{x, u}g_{j}(t, \bar{x}(t), \bar{u}(t)) + (0, -\zeta(t)).
\end{multline}
From \eqref{Eq_Dr_27}, 
\begin{align}
-\lambda \nabla \ell(\bar{x}(t_{0}), \bar{x}(t_{1})) + (p(t_{0}), -p(t_{1})) \in \mathcal{N}_{C}(\bar{x}(t_{0}), \bar{x}(t_{1})).
\label{Eq_Dr_29}
\end{align}
By \eqref{Eq_Dr_22} and \eqref{Eq_Dr_23}, from the convergences of $\{(\sqa^{\sind},\sqb^{\sind})\}$, $\{(\varphi^\sind,\psi^\sind)\}$, $\{(\sqd^{\sind}, \sqe^{\sind})\}$, $\{(\gamma^\sind,\xi^\sind)\}$, $\{x^{\sind}\}$, $\{u^{\sind}\}$, $\{p^{\sind}\}$, $\{\dot{p}^{\sind}\}$, and $\{\lambda_{\sind}\}$, from the continuity of $\ell$ and $f$, and from the uniqueness of the limits, we have
\begin{equation} \label{eq1}
- \nabla_{x, u}f(t, \bar{x}(t), \bar{u}(t))^{T}p(t) + (-\dot{p}(t), 0) - (\varphi(t),\psi(t)) = (0,0) ~ \text{a.e. in} ~ [t_0,t_1]
\end{equation}
and
\begin{equation} \label{eq2}
-\lambda \nabla \ell(\bar{x}(t_{0}), \bar{x}(t_{1})) + (p(t_{0}), -p(t_{1})) - (\gamma,\xi) = (0,0).
\end{equation}
From \eqref{Eq_Dr_28}, \eqref{Eq_Dr_29}, \eqref{eq1} and \eqref{eq2}, we obtain
\begin{align*}
(\varphi(t), \psi(t)) &= \sum_{i=1}^{m_{b}}q_{i}(t)\nabla_{x, u}b_{i}(t, \bar{x}(t), \bar{u}(t)) + \sum_{j=1}^{m_{g}}r_{j}(t)\nabla_{x, u}g_{j}(t, \bar{x}(t), \bar{u}(t)) \\
&\quad + (0, -\zeta(t)) ~ \text{a.e. in} ~ [t_0,t_1]
\end{align*}
and 
$$
(\gamma,\xi) \in \mathcal{N}_{C}(\bar{x}(t_{0}), \bar{x}(t_{1})).
$$
From $\vert\psi^{\sind}(t)\vert \leq k_{\psi}(t)$ almost everywhere in $[t_0,t_1]$ and the convergence of $\{\psi^{\sind}\}$, we obtain 
$$
\vert\psi(t)\vert \leq k_{\psi}(t) ~ \text{a.e. in} ~ [t_0,t_1]. 
$$
Since by \eqref{Eq_Dr_30}, $\zeta(t) \in \co\mathcal{N}_{U(t)}(\bar{u}(t)) = \mathcal{N}_{U_{\delta}(t)}(\bar{u}(t))$ almost everywhere in $[t_{0}, t_{1}]$, it follows that
\begin{align*}
(\varphi, \psi, \gamma, \xi) \in \mathcal{M}_{\delta}(\bar{x}, \bar{u}, 0)
\end{align*}
and, therefore, $(\bar{x}, \bar{u})$ is an AWMP-regular process.    
\end{proof}

\section{Sufficiency Criteria for AWMP-Regularity} \label{Suff_Crit}

In this section, we propose some criteria for the sufficiency of AWMP regularity. The initial set of criteria is organized in Theorem \ref{Lema_Cond_Sufi_Reg_AWMP}, which is presented subsequently. However, it is imperative to acknowledge the following remark prior to the subsequent statement.

\begin{remark}
\label{Obs_Rev_1}
Remember that, given any sequences $\{(x^{\sind}, u^{\sind}, \sqc^{\sind})\} \subset W_{n}^{1,1} \times L_{m}^{1} \times \mathscr{C}$ and $\{(\varphi^{\sind}, \psi^{\sind}, \gamma^{\sind}, \xi^{\sind})\} \subset L_{n}^{1} \times L_{m}^{1} \times \mathbb{R}^{n} \times \mathbb{R}^{n}$ with $(\varphi^{\sind}, \psi^{\sind}, \gamma^{\sind}, \xi^{\sind}) \in \mathcal{M}_{\delta}(x^{\sind}, u^{\sind}, \sqc^{\sind})$ for all $\sind \in \mathbb{N}$, by the definition of $\mathcal{M}_{\delta}(x^{\sind}, u^{\sind}, \sqc^{\sind})$, there exist $k_{\psi} \in L_{m}^{1}$ and sequences $\{(q^{\sind}, r^{\sind})\} \subset L_{m_{b}}^{1} \times L_{m_{g}}^{1}$ and $\{\zeta^{\sind}\} \subset L_{m}^{1}$ such that, for each $\sind \in \mathbb{N}$,
\begin{align}
& (\varphi^{\sind}(t), \psi^{\sind}(t)) = \sum_{i=1}^{m_{b}} q_{i}^{\sind}(t) \nabla_{x, u} b_{i}(t, x^{\sind}(t), u^{\sind}(t))  \label{cond1} \\
& \qquad\qquad\qquad\quad + \sum_{j=1}^{m_{g}} r_{j}^{\sind}(t) \nabla_{x, u} g_{j}(t, x^{\sind}(t), u^{\sind}(t))  \nonumber \\
& \qquad\qquad\qquad\quad + (0, -\zeta^{\sind}(t)) ~ \text{a.e.~in} ~ [t_0,t_1], \nonumber \\
& \zeta^{\sind}(t) \in \mathcal{N}_{U_{\delta}(t)}(u^{\sind}(t)) ~ \text{a.e.~in} ~ [t_0,t_1], \label{cond2} \\ 
& r_{j}^{\sind}(t)g_{j}^{-}(t, x^{\sind}(t), u^{\sind}(t)) = \sqc_{j}^{\sind}(t), ~ r_{j}^{\sind}(t) \leq 0 ~ \text{a.e.~in} ~ [t_0,t_1], ~ j=1,\ldots,m_g, \label{cond3} \\
& (\gamma^{\sind}, \xi^{\sind}) \subset \mathcal{N}_{C}(x^{\sind}(t_{0}), x^{\sind}(t_{1})), \label{cond4} \\
& |\psi^{\sind}(t)| \leq k_{\psi}(t) ~ \text{a.e.~in} ~ [t_0,t_1]. \label{cond5} 
\end{align}

In general, the sequences $\{(q^{\sind}, r^{\sind})\} \subset L_{m_{b}}^{1} \times L_{m_{g}}^{1}$ and $\{\zeta^{\sind}\} \subset L_{m}^{1}$ may be unbounded even if $\{(x^{\sind}, u^{\sind}, \sqc^{\sind})\}$ and $\{(\varphi^{\sind}, \psi^{\sind}, \gamma^{\sind}, \xi^{\sind})\}$ converge.
\end{remark}

The following hypothesis will be employed in this section. Let $(\bar{x}, \bar{u})$ be a reference process of (P) and $\delta > 0$.
\begin{itemize}
\item[(H10)] There exists an integrable function $\tilde{c}_{bg}$ such that 
$$
|\nabla_{u}(b,g)(t,x,u)| \leq \tilde{c}_{bg}(t) ~ \forall \, (x,u) \in \Omega_{\delta}(t) ~ \text{a.e.~in} ~ [t_{0}, t_{1}].
$$
\end{itemize}

\begin{theorem} 
\label{Lema_Cond_Sufi_Reg_AWMP}
    Let $(\bar{x}, \bar{u})$ be a feasible process for (P). Suppose that given $\delta > 0$ and  
    \begin{align*}
        (\varphi,\psi,\gamma,\xi) \in \wl\limsup_{\genfrac{}{}{0pt}{}{x \to \bar{x}, u \to \bar{u}}{\sqc \to 0}}\mathcal{M}_{\delta}(x, u, \sqc),
    \end{align*}
    the hypotheses (H3)-(H6) and (H10) are valid and there exist integrable functions $k_{q}$ and $k_{r}$ such that 
    \begin{align} \label{assumption}
    |q^{\sind}(t)| \leq k_{q}(t) \quad \text{and} \quad |r^{\sind}(t)| \leq k_{r}(t) ~ \text{a.e.~in} ~ [t_{0}, t_{1}] ~\forall \, \sind \in \mathbb{N},
    \end{align}
    where $\{(q^{\sind}, r^{\sind})\} \subset L_{m_{b}}^{1}\times L_{m_{g}}^{1}$ is the sequence referred to in Remark \ref{Obs_Rev_1}. Then, $(\bar{x}, \bar{u})$ is an AWMP-regular process.
\end{theorem}
\begin{proof} 
    Let $\delta > 0$ and 
    \begin{align*}
        (\varphi, \psi, \gamma, \xi) \in \wl\limsup_{\genfrac{}{}{0pt}{}{x \to \bar{x}, u \to \bar{u}}{\sqc \to 0}}\mathcal{M}_{\delta}(x, u, \sqc).
    \end{align*}
    By definition, there exist $\{(x^{\sind}, u^{\sind}, \sqc^{\sind})\} \subset W_{n}^{1,1}\times L_{m}^{1}\times \mathscr{C}$ and $\{(\varphi^{\sind}, \psi^{\sind}, \gamma^{\sind}, \xi^{\sind})\} \subset L_{n}^{1}\times L_{m}^{1}\times \mathbb{R}^{n}\times \mathbb{R}^{n}$ such that $x^{\sind} \to \bar{x}$ uniformly, $u^\sind \to \bar{u}$ in $L_{m}^{1}$, $\sqc_{j}^{\sind}(t) \to 0$ almost everywhere in $[t_{0}, t_{1}]$, $j = 1, \dots, m_{g}$, $(\varphi^{\sind}, \psi^{\sind}) \rightharpoonup (\varphi, \psi)$ in $L_{n}^{1}\times L_{m}^{1}$, $(\gamma^{\sind}, \xi^{\sind}) \to (\gamma, \xi)$, and $(\varphi^{\sind}, \psi^{\sind}, \gamma^{\sind}, \xi^{\sind}) \in \mathcal{M}_{\delta}(x^{\sind}, u^{\sind}, \sqc^{\sind})$ for all $\sind \in \mathbb{N}$. According to Remark \ref{Obs_Rev_1}, there exist $k_{\psi} \in L_{m}^{1}$ and sequences $\{(q^{\sind}, r^{\sind})\} \subset L_{m_{b}}^{1}\times L_{m_{g}}^{1}$ and $\{\zeta^{\sind}\} \subset L_{m}^{1}$ obeying (\ref{cond1})--(\ref{cond5}).

Let us define, almost everywhere in $[t_{0}, t_{1}]$, $\sind \in \mathbb{N}$,
\begin{align} 
& \tilde{q}^{\sind}(t) := \dfrac{q^{\sind}(t)}{1 + k_{q}(t) + k_{r}(t)}, \quad \tilde{r}^{\sind}(t) := \dfrac{r^{\sind}(t)}{1 + k_{q}(t) + k_{r}(t)}, \label{DefMultiplicadores1} \\
& \tilde{\zeta}^{\sind}(t) := \dfrac{\zeta^{\sind}(t)}{1 + k_{q}(t) + k_{r}(t)}. \label{DefMultiplicadores2}
\end{align}
It follows from \eqref{assumption} that $\{\tilde{q}^{\sind}\}$ and $\{\tilde{r}^{\sind}\}$ are uniformly bounded. Thus, by extracting subsequences (we do not relabel), we have that $\tilde{q}^{\sind} \stackrel{\star}{\rightharpoonup} \tilde{q}$ and $\tilde{r}^{\sind} \stackrel{\star}{\rightharpoonup} \tilde{r}$ in $L_{m_{b}}^{\infty}$ and $L_{m_{g}}^{\infty}$, respectively. Moreover, from \eqref{cond1} and \eqref{cond3}, for each $\sind \in \mathbb{N}$ and almost everywhere in $[t_{0}, t_{1}]$,
\begin{multline}
\left(\dfrac{\varphi^{\sind}(t)}{1 + k_{q}(t) + k_{r}(t)}, \dfrac{\psi^{\sind}(t)}{1 + k_{q}(t) + k_{r}(t)}\right) \\ 
= \sum_{i=1}^{m_{b}}\tilde{q}_{i}^{\sind}(t)\nabla_{x, u}b_{i}(t, x^{\sind}(t), u^{\sind}(t)) \\ + \sum_{j=1}^{m_{g}}\tilde{r}_{j}^{\sind}(t)\nabla_{x, u}g_{j}(t, x^{\sind}(t), u^{\sind}(t)) + \left(0, \tilde{\zeta}^{\sind}(t)\right),
\label{Eqrev_1}
\end{multline}
and, for $j = 1, \dots, m_{g}$,
\begin{align}
    \tilde{r}_{j}^{\sind}(t)g_{j}^{-}(t, x^{\sind}(t), u^{\sind}(t)) = \sqc_{j}^{\sind}(t)\left(\dfrac{1}{1 + k_{q}(t) + k_{r}(t)}\right), ~ \tilde{r}_{j}^{\sind}(t) \leq 0, 
    \label{Eq_b_2}
\end{align}
with $\sqc_{j}^{\sind}(t) \to 0$ almost everywhere in $[t_{0}, t_{1}]$ for $j = 1, \dots, m_{g}$. 

Using \eqref{cond5}, \eqref{assumption}, \eqref{DefMultiplicadores1}, \eqref{DefMultiplicadores2}, \eqref{Eqrev_1} and (H10), we see that
\begin{align*}
    |\tilde{\zeta}^{\sind}(t)| \leq \dfrac{k_{\psi}(t)}{1 + k_{q}(t) + k_{r}(t)} + 2\tilde{c}_{bg}(t) ~ \text{a.e.~in} ~ [t_0,t_1] ~ \forall \, \sind \in \mathbb{N}.
\end{align*}
By Dunford-Pettis Theorem \cite[Theorem 2.5.1]{Vinter2000}, there exists a subsequence $\{\tilde{\zeta}^{\sind}\}$ (again, we do not relabel) such that $\tilde{\zeta}^{\sind} \rightharpoonup \tilde{\zeta}$ in $L_{m}^{1}$ and
\begin{align*}
    \int_{t_{0}}^{t}\tilde{\zeta}^{\sind}(s)\dif s \to \int_{t_{0}}^{t}\tilde{\zeta}(s)\dif s \quad\text{ uniformly}.
\end{align*} 

From \eqref{Eq_b_2}, for any Lebesgue measurable set $I \subset [t_{0}, t_{1}]$, for $j = 1, \dots, m_{g}$, we have 
\begin{align}
    \nonumber
    \int_{I}\left(\dfrac{\sqc_{j}^{\sind}(t)}{1 + k_{q}(t) + k_{r}(t)}\right)\dif t & 
    %
     = \int_{I}\tilde{r}_{j}^{\sind}(t)[g_{j}^{-}(t, x^{\sind}(t), u^{\sind}(t)) - g_{j}^{-}(t, \bar{x}(t), \bar{u}(t))]\dif t  \nonumber \\
    & \quad + \int_{I}\tilde{r}_{j}^{\sind}(t)g_{j}^{-}(t, \bar{x}(t), \bar{u}(t)) \dif t
    \label{Eq_Revis_1} 
\end{align}
and
\begin{align}
    \int_{I}\tilde{r}_{j}^{\sind}(t)\dif t \leq 0 ~ \forall \, \sind \in \mathbb{N}.
    \label{Eq_Revis_2}
\end{align}
From \eqref{assumption}, \eqref{DefMultiplicadores1}, \eqref{Eq_b_2}, (H5) and (H6), we also have that there exists an integrable function $\hat{c}_{bg}$ such that, for $j = 1, \dots, m_{g}$ and almost everywhere in $[t_0,t_1]$,
\begin{align*}
\left\vert \sqc_{j}^{\sind}(t)\left(\dfrac{1}{1 + k_{q}(t) + k_{r}(t)}\right) \right\vert = \left\vert \tilde{r}_{j}^{\sind}(t) \right\vert \left\vert g_{j}^{-}(t, x^{\sind}(t), u^{\sind}(t)) \right\vert \leq \hat{c}_{bg}(t),
\end{align*}
and
\begin{align*}
\left\vert \tilde{r}_{j}^{\sind}(t)[g_{j}^{-}(t, x^{\sind}(t), u^{\sind}(t)) - g_{j}^{-}(t, \bar{x}(t), \bar{u}(t))] \right\vert \leq \delta k_{bg}(t).
\end{align*}
Additionally, we have that $x^{\sind} \to \bar{x}$ uniformly and $u^{\sind}(t) \to \bar{u}(t)$ almost everywhere in $[t_{0}, t_{1}]$ after a subsequence extraction (we do not relabel, as always), as a consequence of $u^\sind \to \bar{u}$ in $L_{m}^{1}$. We also have that $\sqc_{j}^{\sind}(t) \to 0$ almost everywhere in $[t_{0}, t_{1}]$, $j = 1, \dots, m_{g}$, and $\tilde{r}^{\sind} \stackrel{\star}{\rightharpoonup} \tilde{r}$ in $L_{m_{g}}^{\infty}$. Thus, taking the limit as $\sind \to \infty$ in \eqref{Eq_Revis_1} and \eqref{Eq_Revis_2}, by applying the Dominated Convergence Theorem (see \cite[Theorem 3.25]{Gordon1994}, for example) in the integral in the left hand side of \eqref{Eq_Revis_1} and in the first integral in the right hand side of \eqref{Eq_Revis_1}, using the fact that $g(t, \cdot, \cdot)$ is continuous almost everywhere in $[t_{0}, t_{1}]$ and the weak$^\star$ convergence of $\{\tilde{r}^{\sind}$\}, we obtain
\begin{align} \label{Eq_b_5}
\tilde{r}_{j}(t)g_{j}^{-}(t, \bar{x}(t), \bar{u}(t)) = 0 ~ \text{and} ~ \tilde{r}_{j}(t) \leq 0 ~ \text{a.e~in} ~ [t_{0}, t_{1}], ~ j = 1, \dots, m_{g}.
\end{align}

Similarly, note that for any Lebesgue measurable set $\hat{I} \subset [t_{0}, t_{1}]$, using \eqref{Eqrev_1}, we obtain
\begin{align*}
	&\int_{\hat{I}}\varphi^{\sind}(t)\left(\dfrac{1}{1 + k_{q}(t) + k_{r}(t)}\right)\dif t \\
    &= \int_{\hat{I}}\sum_{i=1}^{m_{b}}\tilde{q}_{i}^{\sind}(t)\left[\nabla_{x}b_{i}(t, x^{\sind}(t), u^{\sind}(t)) - \nabla_{x}b_{i}(t, \bar{x}(t), \bar{u}(t))\right]\dif t\\
	&\quad + \int_{\hat{I}}\sum_{j=1}^{m_{g}}\tilde{r}_{j}^{\sind}(t)\left[\nabla_{x}g_{j}(t, x^{\sind}(t), u^{\sind}(t)) - \nabla_{x}g_{j}(t, \bar{x}(t), \bar{u}(t))\right]\dif t\\
	&\quad + \int_{\hat{I}}\sum_{i=1}^{m_{b}}\tilde{q}_{i}^{\sind}(t) \nabla_{x}b_{i}(t, \bar{x}(t), \bar{u}(t))\dif t + \int_{\hat{I}}\sum_{j=1}^{m_{g}}\tilde{r}_{j}^{\sind}(t) \nabla_{x}g_{j}(t, \bar{x}(t), \bar{u}(t))\dif t
\end{align*}
and
\begin{align*}
	&\int_{\hat{I}}\left(\dfrac{\psi^{\sind}(t)}{1 + k_{q}(t) + k_{r}(t)}\right)\dif t \\
    &= \int_{\hat{I}}\sum_{i=1}^{m_{b}}\tilde{q}_{i}^{\sind}(t)\left[\nabla_{u}b_{i}(t, x^{\sind}(t), u^{\sind}(t)) - \nabla_{u}b_{i}(t, \bar{x}(t), \bar{u}(t))\right]\dif t\\
	&\quad + \int_{\hat{I}}\sum_{j=1}^{m_{g}}\tilde{r}_{j}^{\sind}(t)\left[\nabla_{u}g_{j}(t, x^{\sind}(t), u^{\sind}(t)) - \nabla_{u}g_{j}(t, \bar{x}(t), \bar{u}(t))\right]\dif t\\
	&\quad + \int_{\hat{I}}\sum_{i=1}^{m_{b}}\tilde{q}_{i}^{\sind}(t) \nabla_{u}b_{i}(t, \bar{x}(t), \bar{u}(t))\dif t  + \int_{\hat{I}}\sum_{j=1}^{m_{g}}\tilde{r}_{j}^{\sind}(t) \nabla_{u}g_{j}(t, \bar{x}(t), \bar{u}(t))\dif t \\ 
	& \quad - \int_{\hat{I}}\tilde{\zeta}^{\sind}(t)\dif t
\end{align*}
for each $\sind \in \mathbb{N}$. Thus, applying the limit as $\sind \to \infty$, using the uniform convergence of $\{x^{\sind}\}$, the pointwise convergence almost everywhere in $[t_{0}, t_{1}]$ of $\{u^{\sind}\}$, the weak convergence of $\{(\zeta^{\sind}, \varphi^{\sind}, \psi^{\sind})\}$ in $L_{m}^{1}\times L_{n}^{1}\times L_{m}^{1}$ and the weak$^\star$ convergence of $\{(\tilde{q}^{\sind}, \tilde{r}^{\sind})\}$ in $L_{n}^{\infty}\times L_{m}^{\infty}$, along with the continuity of $\nabla_{x,u}(b, g)(t, \cdot, \cdot)$ almost everywhere in $[t_{0}, t_{1}]$ (see Hypothesis (H5)), the uniform boundedness of $\{\tilde{q}^{\sind}\}$ and $\{\tilde{r}^{\sind}\}$, the fact that the gradients vectors of $b_i(t,\cdot,\cdot)$, $i=1,\ldots,m_b$, and $g_j(t\cdot,\cdot)$, $j=1,\ldots,m_g$, are locally uniformly integrably bounded almost everywhere in $[t_{0}, t_{1}]$ (see Hypothesis (H5)), and the Dominated Convergence Theorem \cite[Theorem 3.25, p. 45]{Gordon1994}, it follows that
\begin{align}
	\nonumber
	& \left(\dfrac{\varphi(t)}{1 + k_{q}(t) + k_{r}(t)}, \dfrac{\psi(t)}{1 + k_{q}(t) + k_{r}(t)}\right) \\ 
	& \qquad = \sum_{i=1}^{m_{b}}\tilde{q}_{i}(t)\nabla_{x, u}b_{i}(t, \bar{x}(t), \bar{u}(t)) + \sum_{j=1}^{m_{g}}\tilde{r}_{j}(t)\nabla_{x, u}g_{j}(t, \bar{x}(t), \bar{u}(t)) \nonumber \\
	& \qquad\quad~ + \left(0, -\tilde{\zeta}(t)\right)
\label{Eq_b_6}
\end{align}
almost everywhere in $[t_0,t_1]$.

Defining
\begin{align*}
& q(t) := (1 + k_{q}(t) + k_{r}(t)) \tilde{q}(t) ~ \text{a.e.~in} ~ [t_{0}, t_{1}], \\ 
& r(t) := (1 + k_{q}(t) + k_{r}(t)) \tilde{r}(t) ~ \text{a.e.~in} ~ [t_{0}, t_{1}], \\
& \hat{\zeta}(t) := (1 + k_{q}(t) + k_{r}(t)) \tilde{\zeta}(t) ~ \text{a.e.~in} ~ [t_{0}, t_{1}],
\end{align*}
from \eqref{Eq_b_5} and \eqref{Eq_b_6}, we have that 
$$
r_{j}(t)g_{j}^{-}(t, \bar{x}(t), \bar{u}(t)) = 0, ~ r_{j}(t) \leq 0 ~ \text{a.e.~in} ~ [t_{0}, t_{1}], ~ j=1,\ldots,m_g,
$$
and
\begin{align*}
	(\varphi(t), \psi(t)) = { } & \sum_{i=1}^{m_{b}}q_{i}(t)\nabla_{x, u}b_{i}(t, \bar{x}(t), \bar{u}(t)) \\ & + \sum_{j=1}^{m_{g}}r_{j}(t)\nabla_{x, u}g_{j}(t, \bar{x}(t), \bar{u}(t)) + (0, - \hat{\zeta}(t)) 
\end{align*}
almost everywhere in $[t_0,t_1]$.

Finally, invoking the properties of upper semicontinuity of the limiting normal cones, we conclude that $\hat{\zeta}(t) \in \mathcal{N}_{U_{\delta}(t)}(\bar{u}(t))$ almost everywhere in $[t_{0}, t_{1}]$ and $(\gamma, \xi) \in \mathcal{N}_{C}(\bar{x}(t_0), \bar{x}(t_1))$. Therefore,
\begin{align*}
	(\varphi, \psi, \gamma, \xi) \in \mathcal{M}_{\delta}(\bar{x}, \bar{u}, 0)
\end{align*}
and, consequently, $(\bar{x}, \bar{u})$ is an AWMP-regular process.
\end{proof}

The Calibrated Constraint Qualification (CCQ) was introduced by Clarke and de Pinho in \cite{Clarke2010} for problems more general than (P), including, for example, non-smooth mixed constraints. It has been demonstrated that the weak maximum principle holds under CCQ. It is important to note that in \cite{Clarke2010}, a more general notion of local minimizer was employed than the one utilized in this study. This more general notion incorporates a radius function, $R$, which is required to be measurable. We hereby introduce an asymptotic version of CCQ, which will also provide a sufficient condition for AWMP-regularity.

\begin{definition} \label{ACCQ}
The \emph{Asymptotic Calibrated Constraint Qualification} (ACCQ) is satisfied at a feasible process $(\bar{x},\bar{u})$ if there exist $\delta > 0$ and an integrable function $M: [t_{0}, t_{1}] \to \mathbb{R}$, such that, given sequences $\{(x^{\sind}, u^{\sind})\}\subset W_{n}^{1,1} \times L_{m}^{1}$, $\{(q^{\sind}, r^{\sind})\} \subset L_{m_{b}}^{1} \times L_{m_{g}}^{1}$ and $\{\zeta^{\sind}\} \subset L_{m}^{1}$ with
\begin{align*}
& x^{\sind} \rightarrow \bar{x} ~ \text{uniformly}, \\
& u^{\sind} \to \bar{u} ~ \text{in} ~ L_{m}^{1}, \\
& \zeta^{\sind}(t) \in \mathcal{N}_{U_{\delta}(t)}(u^{\sind}(t)) ~ \text{a.e.~in} ~ [t_0,t_1], \\ 
& r_{j}^{\sind}(t)g_{j}^{-}(t, x^{\sind}(t), u^{\sind}(t)) \to 0, ~ r_{j}^{\sind}(t) \leq 0 ~ \text{a.e.~in} ~ [t_0,t_1], ~ j=1,\ldots,m_g,
\end{align*}
for $\{(\varphi^{\sind}, \psi^{\sind})\} \subset L_{n}^{1}\times L_{m}^{1}$ satisfying, almost everywhere in $[t_0,t_1]$,
\begin{align*}
(\varphi^{\sind}(t), \psi^{\sind}(t)) &= \sum_{i=1}^{m_{b}} q_{i}^{\sind}(t) \nabla_{x, u} b_{i}(t, x^{\sind}(t), u^{\sind}(t)) + \sum_{j=1}^{m_{g}} r_{j}^{\sind}(t) \nabla_{x, u} g_{j}(t, x^{\sind}(t), u^{\sind}(t)) \\
&\quad + (0, -\zeta^{\sind}(t)),
\end{align*}
it holds that
\begin{align*}
|(q^{\sind}(t), r^{\sind}(t))| \leq M(t)|\psi^{\sind}(t)| ~ \text{a.e.~in} ~ [t_0,t_1] ~ \forall \sind \in \mathbb{N}.
\end{align*}
\end{definition}

It is evident that if ACCQ is valid, then CCQ is also valid. In the sequel, the validity of ACCQ at a reference process is demonstrated to ensure AWMP-regularity.
\begin{theorem} 
\label{Teo_ACCQ_AWMP_reg}
    Assume that ACCQ, (H3)--(H6) and (H10) are satisfied at a feasible process $(\bar{x},\bar{u})$ with the same $\delta > 0$. Then $(\bar{x}, \bar{u})$ is an AWMP-regular process.
    \end{theorem}
    \begin{proof}
    Let us consider a $\delta > 0$ for which Definition \ref{ACCQ} and all other assumptions hold. Take
    \begin{align*}
        (\varphi, \psi, \gamma, \xi) \in \wl\limsup_{\substack{x \to \bar{x}, u \to  \bar{u} \\ \sqc \to 0}}\mathcal{M}_{\delta}(x, u, \sqc).
    \end{align*}
    Thus, by definition, there exist sequences $\{(x^{\sind}, u^{\sind}, \sqc^{\sind})\} \subset W_{n}^{1,1} \times L_{m}^{1} \times \mathscr{C}$ and $\{(\varphi^{\sind}, \psi^{\sind}, \gamma^{\sind}, \xi^{\sind})\} \subset L_{n}^{1} \times L_{m}^{1} \times \mathbb{R}^{n} \times \mathbb{R}^{n}$ such that $x^{\sind} \to \bar{x}$ uniformly, $u^{\sind} \to \bar{u}$ in $L_{m}^{1}$, $\sqc_{j}^{\sind}(t) \to 0$ almost everywhere in $[t_{0}, t_{1}]$, $j = 1, \dots, m_{g}$, $(\varphi^{\sind}, \psi^{\sind}) \rightharpoonup (\varphi, \psi)$, $(\gamma^{\sind}, \xi^{\sind}) \to (\gamma, \xi)$, and
    \begin{align*}
        (\varphi^{\sind}, \psi^{\sind}, \gamma^{\sind}, \xi^{\sind}) \in \mathcal{M}_{\delta}(x^{\sind}, u^{\sind}, \sqc^{\sind}) ~ \forall \, \sind \in \mathbb{N}.
    \end{align*}
    From the inclusion above, we know that there exist $k_{\psi} \in L_{m}^{1}$ and sequences $\{(q^{\sind}, r^{\sind})\}$ $\subset L_{m_{b}}^{1} \times L_{m_{g}}^{1}$ and $\{\zeta^{\sind}\} \subset L_{m}^{1}$ such that (\ref{cond1})--(\ref{cond5}) are valid. From (\ref{cond1})--(\ref{cond3}), since ACCQ holds, we know that there exists $M: [t_{0}, t_{1}] \to \mathbb{R}$ integrable such that 
        \begin{align*}
            |(q^{\sind}(t), r^{\sind}(t))| \leq  M(t)|\psi^{\sind}(t)| ~ \text{a.e~in} ~ [t_0,t_1] ~ \forall \, \sind \in \mathbb{N}.
        \end{align*}
From (\ref{cond5}), we get
$$
|(q^{\sind}(t), r^{\sind}(t))| \leq M(t)k_{\psi}(t) ~ \text{a.e~in} ~ [t_0,t_1] ~ \forall \, \sind \in \mathbb{N}.
$$        
We see that \eqref{assumption} holds. It follows from Theorem \ref{Lema_Cond_Sufi_Reg_AWMP} that $(\bar{x}, \bar{u})$ is AWMP-regular, as we aimed to prove.
\end{proof}

\section{Conclusions} \label{Sec_Concl}

In summary, the present work contributes to the theory of sequential optimality conditions in optimal control by introducing the AWMP-regularity condition. AWMP-regularity, understood as the minimal requirement under which the asymptotic weak maximum principle imply the classical weak maximum principle, provides a theoretical foundation that strengthens the connection between sequential optimality conditions and classical optimality conditions in optimal control. This constraint qualification finds inspiration from analogous advances in nonlinear programming, namely AKKT-regularity, Cone Continuity Property (CCP), and AM-regularity conditions. However, it is formulated to align with the unique structure of optimal control problems.

The AWMP-regularity condition provides a systematic method for validating optimal candidate solutions under AWMP conditions using the rigor of classical conditions. This expands the applicability and interpretive power of sequential methods in optimal control. The AWMP-regularity condition offers a theoretical guarantee of the weak maximum principle's validity and enhances the robustness of solution methods. This includes practical algorithms like the method of multipliers presented in \cite{Moreira2024art}, which naturally generate limit points satisfying AWMP conditions.

An asymptotic version of the CCQ for optimal control problems with smooth mixed constraints is also presented, based on the original work in \cite{Clarke2010}. This reformulation gives rise to another constraint qualification, called ACCQ. Furthermore, it is demonstrated that if ACCQ holds, then AWMP-regularity is satisfied, as formalized in Theorem \ref{Teo_ACCQ_AWMP_reg}.

In future research, the aim is to examine the relationships between additional constraint qualifications and AWMP regularity. These additional constraint qualifications may include full rank type CQs \cite{Hestenes1966, Neustadt1976, Osmolovskii1975, dePinho2002, dePinho2003}, constant rank type CQs \cite{Andreani2020, Pereira2020}, and the CCQ itself, among others found in the literature, as mentioned in the introduction.

\section*{Acknowledgments}

This study was financed, in part, by the São Paulo Research Foundation (FAPESP) and the National Council for Scientific and Technological Development (CNPq), Brazil. Process Numbers 2022/16005-0 and 305245/2024-4. This work was also supported by the National Council for the Improvement of Higher Education (CAPES/UNESP/IBILCE). Finance Code 001.

\bibliographystyle{plain}
\bibliography{references}
\end{document}